\documentclass[11pt,leqno]{article}
\usepackage{amsmath, amsfonts, theorem,color, 
}
\usepackage{latexsym}
\usepackage{graphicx}
\makeatletter
\def\@maketitle{\newpage
    \null
    \vskip .8truein
    \begin{center}%
     {\bf \@title \par}%
     \vskip 1.5em
     {\small
      \lineskip .5em
      \begin{tabular}[t]{c}\@author
      \end{tabular}\par}%
    \end{center}%
    \par
    \vskip .4truein}
\@addtoreset{equation}{section} \@addtoreset{theorem}{section}
\@addtoreset{lemma}{section} \@addtoreset{proposition}{section}
\@addtoreset{definition}{section} \@addtoreset{corollary}{section}
\@addtoreset{remark}{section}

\let\d=\delta
\let\D=\Delta

\let\nn=\nonumber

\newcommand{\re}{{\mathbb R}}

\let\ds=\displaystyle

\let\d=\delta

\newtheorem{theorem}{Theorem}[section]
\newtheorem{lemma}{Lemma}[section]
\newtheorem{proposition}{Proposition}[section]

\newtheorem{remark}{Remark}[section]

\newtheorem{claim}{Claim}[section]
  {\hfill$\Box$\bigskip\par}

\def\proof{\list{}{\setlength{\leftmargin}{0pt}
                      \parskip=0pt\parsep=0pt\listparindent=2em
                      \itemindent=0pt}\item[]\futurelet\testchar\@maybe}

\def\@maybe{\ifx[\testchar \let\next\@Opt
          \else \let\next\@NoOpt \fi \next}
\def\@Opt[#1]{{\it Proof of #1.\ }}\def\@NoOpt{{\it Proof.\ }}
\begin{document}
\title{\Large \bf Singular perturbations for a subelliptic operator}

\author{{\large \sc Paola Mannucci, Claudio Marchi, Nicoletta Tchou}\\
 \rm Universit\`a degli Studi di Padova, Universit\'e de Rennes 1}
%%%%%%%%%%%%%%%%%%%%%%%%%%%%%%%%%%%%%%%%%%%%%%%%%%%%%%%%%%%%%%%%%
\maketitle

%\begin{abstract}
\begin{abstract}
\noindent 
We study some classes of singular perturbation problems where the dynamics of the fast variables evolve in the whole space obeying to an infinitesimal operator which is subelliptic and ergodic. We prove that the corresponding ergodic problem admits a solution which is  globally Lipschitz continuous and it has at most a  logarithmic growth at infinity.

The main result of this paper establishes that as $\epsilon \rightarrow 0$, the value functions of the singular perturbation problems converge locally uniformly to the solution of an effective problem whose operator and data are explicitly given in terms of the invariant measure for the ergodic operator.
\end{abstract}
\footnote{\date{\today}}
%

%\noindent  {\bf 2010 AMS Subject classification:} 49L25, 35B27, 35J70, 35H10.

%\newpage
%%%%%%%%%%%%%%%%%%%%%%%%%%%%%%%%%%%%%

\section{Introduction}
This paper is devoted to the asymptotic behaviour as $\epsilon \to 0$ of stochastic control systems of the form
\begin{eqnarray*}
dX_t&=&\tilde\phi (X_t,Y_t,u_t)dt+\sqrt2 \tilde \sigma(X_t,Y_t,u_t)dW_t,\qquad X_0=x\in \re^n\\
dY_t&=& \frac1\epsilon b(Y_t)dt+\frac{\sqrt2}{\sqrt\epsilon}\sigma(Y_t) dW_t,\qquad Y_0=y\in\re^m
\end{eqnarray*}
where $u_t$ is a control law, $W_t$ is a standard Brownian motion, while the coefficients $\tilde \phi$, $\tilde \sigma$, $b$ and $\sigma$ are Lipschitz continuous uniformly in $u$.
We are mostly interested in the asymptotic behaviour of the value function
\begin{equation*}
V^\epsilon(t,x,y):=\sup_{u \in{\cal U}} \mathbb E[\int_t^Tf(X_s,Y_s,u_s)ds+ e^{a(t-T)}g(X_T)]
\end{equation*}
where $\mathbb E$ denotes the expectation, $\cal U$ is the set of progressively measurable processes with values in a compact metric set ~$U$ and $a$ is a fixed positive parameter.
Our aim will be to characterize the limit of~$V^\epsilon$ as the solution to an {\it effective} Cauchy problem whose operator and terminal data need to be suitably chosen.
\\
Problems of this type arise from models where the variables $Y$ evolve much faster than the variables $X$.\\
%
%
%Our investigations have been inspired by the models of pricing and trading derivative securities $X_t$ in financial markets with stochastic volatility $Y_t$. In such markets the asset prices are affected by correlated phenomena, modeled as diffusion processes.
%In the framework of mathematical finance, path-dependent models, such as Asian options, involve degenerate
%diffusion processes $X_t$ ; (see [\cite{MR1357666},\cite {MR1768877},\cite{MR1830951}] and the references therein).
%On the other hand our interest on degeneracy of the stochastic equation satisfied by the volatility $Y_t$ comes from some models as Feller or Cox-Ingersoll-Ross one (see [23] pag.42) .
%
%{\color{red}
Although the present work is not directly concerned with financial mathematics, it has been partially inspired by some models of financial assets whose price $X_t$ is a stochastic process with a possibly degenerate diffusion. In such models, the volatility $Y_t$ is also a stochastic process which is correlated to the former one. Some of the financial models discussed in (\cite{MR1357666},\cite {MR1768877},\cite{MR1830951} and the references therein) involve stochastic processes $X_t$ with degenerate diffusions.
On the other hand, the volatility $Y_t$ may also be a stochastic process with a degenerate diffusion, see for example the models of Feller  and  Cox-Ingersoll-Ross (\cite [pag.42]{MR1768877}).
%}
\\
%We investigate the limit of the value function
%\begin{equation*}
%V^\epsilon(t,x,y):=\sup_{u \in{\cal U}} \mathbb E[\int_t^Tf(X_s,Y_s,u_s)ds+ e^{a(t-T)}g(X_T)]
%\end{equation*}
%where $\mathbb E$ denotes the expectation, $\cal U$ is the set of progressively measurable processes with values in a compact metric set ~$U$ and $a$ is a fixed positive parameter.
%
%Our aim will be to characterize the limit of~$V^\epsilon$ as the solution to an {\it effective} Cauchy problem whose operator and terminal data need to be suitably chosen.
The main issue of this paper is to tackle this problem when the coefficients are not periodic in~$y$ and the diffusion matrices $\tilde \sigma \tilde\sigma^T$ and $\sigma \sigma^T$ may be degenerate and unbounded.
For the sake of simplicity, we shall focus our attention to the model case where $\tilde \sigma \tilde\sigma^T$ is bounded and the diffusion matrix~$\sigma$ is the one associated to the Heisenberg group in~$\re^3$
%$$\sigma(y)=\begin{matrix}
%	1 &0\\
%0& 1\\
%2y_2& -2y_1
%\end{matrix}\qquad \textrm{for }y=(y_1, y_2, y_3).
%$$
%
$$\sigma(y)=
\left(\begin{array}{cc}1 & 0 \\0 & 1 \\2y_2 & -2y_1\end{array}\right), \qquad \textrm{for }y=(y_1, y_2, y_3).$$
We note that $\sigma \sigma^T$ is degenerate and with unbounded coefficients.

It is well known that the ergodicity of the fast variable cannot be expected for general drift~$b$. In order to overcome this issue, we consider a drift $b$ in the Ornstein-Uhlenbeck form 
$$b(y)=-(k_1y_1, k_2y_2,k_3y_3) \qquad \textrm{for some }k_1>4, \,k_2>4, k_3>0$$
(see  A2)).
This choice of the drift is reminiscent of other similar conditions about recurrence of diffusion processes in the whole space (see for example \cite{G2} and references therein).

By standard theory (see~\cite{FS}), the value function $V^\epsilon$ is the unique (viscosity) solution to the following Cauchy problem for an Hamilton-Jacobi-Bellman equation
\begin{equation}\label{HJBSP}
\left\{\begin{array}{ll}
-\partial_tV^\epsilon+H\left(x,y,D_xV^\epsilon,D_{xx}^2V^\epsilon,\frac{D_{xy}^2V^\epsilon}{\sqrt\epsilon}\right)&\\
\qquad-\frac1\epsilon{\cal L}(y, D_yV^\epsilon,D_{yy}V^\epsilon)+a V^\epsilon=0 & \textrm{in }(0,T)\times \re^n\times\re^3\\
V^\epsilon(T,x,y)=g(x,y)& \textrm{on }\re^n\times\re^3
\end{array}\right.
\end{equation}
where 
\begin{eqnarray*}
H(x,y,p,X,Z)&:=&\min_{u\in U}\left\{-tr(\tilde \sigma\tilde \sigma^T X)-\tilde \phi\cdot p -2 tr(\tilde \sigma\sigma^T Z)-f(x,y,u) \right\}\\
{\cal L}(y,q,Y)&:=& tr(\sigma\sigma^T Y)+b\cdot q.
\end{eqnarray*}
For the sake of completeness, in order to exhibit the degeneracy and the unboundedness of the operator, we  write explicitly the second order term of ${\cal L}$:
\begin{equation}
\label{calL2nd}
tr(\sigma\sigma^T D^2 U)= U_{y_1y_1}+U_{y_2y_2}+4(y_1^2+y^2_2)U_{y_3y_3}+4y_2U_{y_1y_3}-
4y_1U_{y_2y_3}.
\end{equation}

We assume without any loss of generality that $a$ is strictly positive; actually, for $a
\leq 0$, the function $W^\epsilon (t,x,y)=e^{-A(T-t)}V^\epsilon(t,x,y)$, with $A>-a$, satisfies the same Cauchy problem but with a positive coefficient of the $0$-th order term.

Our aim is to establish that, as $\epsilon\to0^+$, the function~$V^\epsilon$ converges locally uniformly to a function~$V=V(t,x)$ (namely, independent of~$y$) which can be characterized as the unique (viscosity) solution to the effective Cauchy problem
\begin{equation}\label{EFFSP}
\left\{\begin{array}{ll}
-\partial_tV+\overline H\left(x,D_xV,D_{xx}^2V\right)+a V=0 &\quad \textrm{in }(0,T)\times \re^n\\
V(T,x)= \overline g(x)&\quad \textrm{on }\re^n
\end{array}\right.
\end{equation}
where, for every $(x,p,X)$, the effective Hamiltonian $\overline H(x,p,X)$ and the effective terminal datum are given by
\begin{eqnarray}\label{opeff}
\overline H(x,p,X)&:=& \int_{\re^3}H(x,y,p,X,0) d\mu(y)\\ \label{datoeff}
\overline g(x)&:=& \int_{\re^3}g(x,y) d\mu(y)
\end{eqnarray}
and $\mu$ is the invariant measure of the diffusion process with infinitesimal generator~$-\mathcal L$. As a matter of facts, $\overline H(x,p,X)$ is the ergodic constant~$\lambda$ of the cell problem
\begin{equation}\label{cell}
-tr(\sigma(y)\sigma^T(y)D^2 w(y))-b(y)Dw(y)+H(x,y,p,X,0)=\lambda\quad y\in\re^3,
\end{equation}
while $\overline g (x)$ is the constant obtained in the long time behaviour of the parabolic Cauchy problem
\begin{equation*}
\partial_t w'- {\cal L} w'=0\quad\textrm{in }(0,\infty)\times \re^3,\qquad w'(0,y)=g(x,y)\quad\textrm{on } \re^3.
\end{equation*}

There is a large literature on singular perturbation problems: see ~\cite{AB1,K,KP} and references therein.
We shall follow a pure PDE-approach. In this framework, the singular perturbation problems are strictly related to homogenization problems (see also \cite{MS}); Alvarez and Bardi \cite{AB1,AB3} extended to singular perturbation problems with {\it periodic} fast variables the celebrated {\it perturbed test function method} by Evans (see also \cite{ABM} for some cases in hypoelliptic periodic setting).
Let us also recall that, the papers~\cite{BC, BCM, G2} studied singular perturbation problems of uniformly elliptic operators on the whole space.

The novelties of our results is that the variable $Y_t$ is unbounded and the diffusion matrix of the fast variable may be degenerate and unbounded. In other words, the main issues to overcome are the lack of periodicity and the degeneracy of the operator. 
The proof of our main Theorem \ref{maintheor} is not an adaptation to the subelliptic case of some arguments already known in the non degenerate case. Indeed our proof is based on the perturbed test function method suitably adapted with a Lyapunov function. Moreover our techniques 
shed some light on some difficult points in the literature on the whole space.

Let us recall that existence and uniqueness of the ergodic constant~$\lambda$ for~\eqref{cell} (namely, that $\delta u_\delta$ locally converge to~$\lambda$, where $u_\delta$ solves the approximated cell problem \eqref{probcellappross0} below) and the stabilization to a constant have been established in our previous paper~\cite{MMT}. Unfortunately, by the lack of compactness for $y$, these properties seem to be not sufficient for applying the usual semilimits method for the convergence of~$V^\epsilon$. In order to overcome this issue:
\begin{itemize}
\item
 we shall prove that the cell problem admits a corrector~$w$ which is globally Lipschitz continuous and it has at most a logarithmic type growth at infinity;
\item under some additional assumptions we get that the corrector $w$ is $C^{2,\alpha}(\re^3)$.
\item we take advantage of the existence and uniqueness of the invariant measure and a superlinear Lyapunov function for the operator $\cal L$.
\end{itemize}
In our opinion, the proof of the global Lipschitz continuity of the corrector has its own interest because it can be extended to many other operators in unbounded domains.
In this direction, let us quote the papers ~\cite{G2} and \cite{LN} where similar results are obtained for strictly elliptic operators.\\
Moreover the $C^2$-regularity of the corrector is not straightforward because our operator
contains second order horizontal derivatives and Euclidean first derivatives as well and
such a second order part of the operator does not immediately regularize  the first order one.

The paper is organized as follows: in Section 2 we state the perturbation problem and our main convergence result.
Section 3 is devoted to the solution of the cell problem and its properties. 
In Section~\ref{sect:proof}, by means of these result we prove the convergence of $V^{\epsilon}$ to~$V$.

\section{The convergence result}\label{pbsp}

Throughout this paper unless otherwise explicitly stated, we assume
\begin{itemize}
\item[$A1$)] the diffusion matrix~$\sigma$ has the following form: 
\begin{equation*}%\label{matrixH}
\sigma(y)=
\left(\begin{array}{cc}1 & 0 \\0 & 1 \\2y_2 & -2y_1\end{array}\right), \qquad \textrm{for }y=(y_1, y_2, y_3).
\end{equation*}
\item[$A2$)] the drift is $b(y)=-(k_1y_1, k_2y_2,k_3y_3)$ with $k_1>4$, $k_2>4$, $k_3>0$;
\item[$A3$)] the function~$f=f(x,y,u)$ is Lipschitz continuous in~$(x,y)$ uniformly in $u$ and, for some $C_f>0$, it satisfies
\begin{equation*}
|f(x,y,u)|\leq C_f(1+|x|)\qquad \forall (x,y,u)\in \re^n\times \re^3\times U;
\end{equation*}
\item[$A4$)] 
the function~$g$ is continuous in~$(x,y)$ and there exits $C_g$ such that
\begin{equation*}
|g(x,y)|\leq C_g(1+|x|) \qquad \forall (x,y)\in \re^n\times\re^3;
\end{equation*}
\item[$A5$)] $\tilde \phi(x,y,u)$ and $\tilde\sigma(x,y,u)$ are Lipschitz continuous and bounded in $(x,y)$ uniformly on $u$:
 $\vert\tilde \phi(x,y,u)\vert\leq C_{\tilde \phi}$,  $\vert\tilde \sigma(x,y,u)\vert\leq C_{\tilde\sigma}$;
 \item[$A6)$]
 for any $(x, p, X)$ the function 
 $F(y)=-H(x, y, p, X,0)$ is such that $F$,
 $\frac{\partial F}{\partial y_3}$ and $\frac{\partial^2 F}{\partial y_3^2} $ are bounded and globally Lipschitz.
\end{itemize}
\begin{remark}
\ 
{\rm
\begin{itemize}
\item Let us underline that condition (A2) is linked with the form of the second order operator given in \eqref{calL2nd}. This assumption will play a crucial role in several main points, for instance in \eqref{sublindelta}, Theorem \ref{lipschitzNuovo}, Theorem \ref{regw} and in \eqref{claimcontr1}.
\item We note also that there holds:
\begin{multline}\label{regH1}
\vert H(x,y,p, X,Z)-H(x',y,p', X',Z)\vert\\ \leq C\left(\vert p-p'\vert+\vert X-X'\vert\right)+C\vert x-x'\vert\left(1+\vert p \vert +\vert X\vert\right).
\end{multline}
\item We stress that assumption A6) will be only used for obtaining the regularity of the corrector.
\end{itemize}}
\end{remark}

We state now that the problem \eqref{HJBSP} is well posed and the solution $V^\epsilon$ has a sublinear growth in the slow variable.
\begin{proposition}\label{Existence}
Under Assumptions ($A1$)-($A5$), for any $\epsilon>0$ there exists a unique continuous viscosity solution~$V^\epsilon$ to problem \eqref{HJBSP} such that
\begin{equation}\label{usublinear}
|V^{\epsilon}(t,x,y)|\leq C_0(1+|x|), \quad \forall (t, x, y)\in (0,T)\times\re^n\times \re^n
\end{equation}
for some positive constant~$C_0$ independent on $\epsilon$.
In particular $\{V^{\epsilon}\}_\epsilon$ is a family of locally equibounded functions.
\end{proposition}
\begin{proof}
The uniqueness follows from the comparison principle proved in Da Lio-Ley \cite{DLL} (recall that they require that the diffusion matrix and the drift grow at most quadratically and respectively linearly with respect to the state).
We now claim that there exist a supersolution~$w^+$ and a subsolution~$w^-$ such that $|w^\pm(x)|\leq C(1+|x|)$ for $|x|$ sufficiently large. We shall prove the existence of~$w^+$ and we shall omit the analogous arguments for~$w^-$.\\
Let $w_0\in C^\infty(\re^n)$ be a function in~$x$ such that 
$$w_0=C_1(1+|x|)\quad \textrm{for }|x|\geq R\geq 1,\qquad w_0(x)\geq g(x,y)\quad \forall (x,y),$$
for some positive constants $C_1$ and $R$.
For $|x|\geq R$, there holds
\begin{multline*}
-\partial_t w_0+H(x,y,D_xw_0,D_{xx}^2w_0,\frac{D_{xy}^2w_0}{\sqrt\epsilon})-\frac1\epsilon{\cal L}(y, D_yw_0,D_{yy}w_0)+a w_0=\\
H(x, y,Dw_0, D^2w_0,0)+a w_0=\\
\min_{u\in U}\left\{-tr(\tilde \sigma\tilde \sigma^T D^2w_0)-\tilde \phi\cdot Dw_0-f(x,y) \right\}+aw_0(x)\geq\\
-C_2 C_1+aC_1(1+|x|)-C_f(1+|x|)
\end{multline*}
where  $C_2$ depends on $C_{\tilde \phi}$ and $C_{\tilde\sigma}$. Choosing $C_1$ big enough (for instance $C_1\geq \frac{C_f+1}{a}$) and $R$ big enough (for instance $R\geq C_2 C_1$) the function~$w_0$ is a supersolution \eqref{HJBSP} for $|x|>R$. Eventually adding a new positive constant $C_3$ big enough,
$w^+=w_0+C_3$ is a supersolution in the whole~$(0,T)\times \re^n\times\re^3$ which amounts to our claim.
In conclusion, applying Perron's method, we infer the existence of a solution to~\eqref{HJBSP} verifying \eqref{usublinear}.
\end{proof}

The main purpose of this paper is to prove the following
\begin{theorem}
\label{maintheor}
Under assumptions A1)-A6),  the solution
$V^\epsilon$ of \eqref{HJBSP} converges uniformly on the compact subsets of $(0,T)\times \re^n\times\re^3$
to the unique viscosity solution $V$ of \eqref{EFFSP} where $\overline H$ and $\overline g$ are defined in \eqref{opeff} and respectively in \eqref{datoeff}.
\end{theorem}

\section{The cell problem}
In this section we prove that there exists an unique constant $\lambda$ such that the cell problem~\eqref{cell} admits solutions. We shall also prove the existence of a solution~$w$ which is globally Lipschitz continuous  and with $\log$-growth at infinity. 
Assuming also A6) we prove that $w\in C^2$.
This solution~$w$ will play a crucial role in the proof of Theorem~\ref{maintheor}.
\subsection{Approximated cell problems}
In order to solve the cell problem \eqref{cell}, it is expedient to introduce the approximated problems
\begin{equation}\label{probcellappross0}
\delta u_{\delta}-tr(\sigma(y)\sigma^T(y)D^2 u_{\delta})-b(y)Du_{\delta}=F(y)\qquad \textrm{in }\re^3,
\end{equation}
where  $\delta>0$ and $F(y)=-H(x,y,p,X,0)$ with $(x,p,X)$ fixed.
In this section the results are obtained for a general function $F(y)$ which satisfies:
\begin{equation}\label{F}
F(y) \mbox{ is continuous and bounded in } \re^3. 
\end{equation}
Note that under assumptions (A1)-(A5), for $(x, p, X)$ fixed, the function $F(\cdot)=-H(x,\cdot,p,X,0)$, satisfies assumption~\eqref{F}.

Let us recall from \cite{MMT} some properties of the operator $\cal L$ and functions $u_\delta$ (and 
we refer the reader to this paper for the detailed proof).
\begin{lemma}
\label{eqmu}
There exists a unique invariant measure $\mu$ associated to the operator $-\mathcal L$; moreover
\[
{\cal L}^*\mu=0, \quad \mu>0,\quad \mu\in C^\infty(\re^3)
\]
where ${\cal L}^*$ is the adjoint operator of $-\cal L$.
\end{lemma}
\begin{remark}
As a byproduct of \cite{MMT}, we have the following estimate on the decay of $\mu$ at infinity:
\[\int_{\re^3} (y_1^4+y_2^4+y_3^2)d\mu(y)<+\infty.\]
Actually, by \cite[Prop. 2.1 (proof)]{MMT}, the function $w(y):=(y_1^4+y_2^4)/12+y_3^2/2$ satisfies \cite[eq. (2.10)]{MMT} with $\phi>k(y_1^4+y_2^4+y_3^2)$ (for $k>0$ sufficiently small). Hence, for this choice of $\phi$, relations \cite[equation (2.14)]{MMT} and \cite[equation (2.18)]{MMT} hold true. Letting $\rho\to 0^+$ we get: $\int_{\re^3} \phi d\mu<+\infty$.
\end{remark}
\begin{lemma}\label{exudelta0}
Under assumptions (A1),(A2),\eqref{F}, there exists an unique smooth solution $u_{\delta}$ of the approximating problem 
(\ref{probcellappross0}) such that
\begin{equation*}%\label{sublin}
|u_{\delta}(y)|\leq \frac C\delta \qquad \forall y\in \re^3
\end{equation*}
for some positive constant~$C$ independent of $\delta$. Moreover
the functions $\delta u_{\delta}$ are locally uniformly H\"older continuous, i.e. there exists $\alpha\in(0,1)$ such that for every compact $K\subset \re^3$ there exists a constant $N$ such that 
\begin{equation*}%\label{holdervdelta}
|\delta u_{\delta}(y_1)-\delta u_{\delta}(y_2)|\leq N|y_1-y_2|^{\alpha}, \quad \forall y_1,y_2\in K,\ \forall \delta\in(0,1).
\end{equation*}
The constant $N$ only depends on $K$ and on the data of the problem (in particular is independent of $\delta$).
\end{lemma}
\begin{theorem}\label{conv1}
%Under assumptions ($A_1$)-($A_3$), 
%for any fixed $\overline x$, $\overline u$, $\overline p$, 
The solution $u_{\delta}$ of  problem (\ref{probcellappross0}) given in Lemma~\ref{exudelta0} satisfies
\begin{equation*}%\label{4.10}
\lim_{\delta\rightarrow 0}\delta u_{\delta}= \int_{\re^3}F(y)d\mu(y),
\end{equation*}
where $\mu$ is the invariant measure of $-{\cal L}$ established in Lemma~\ref{eqmu}.
\end{theorem}

\subsection{Global Lipschitz continuity.}
In this section we derive the global Lipschitz continuity of the solution $u_\delta$ of (\ref{probcellappross0}) from its continuity under the weaker assumption of an at most linear growth of $F$. 
In our opinion, this result has its own interest.
We assume: 
\begin{equation}\label{L}
\left\{\begin{array}{l}
(A1)-(A2)\\
F \mbox{ Lipschitz continuous in } \re^3 \mbox{ with Lipschitz constant }L\\
|F(y)|\leq C_F(|y|+1)\quad \forall y\in \re^3
\end{array}\right.
\end{equation}

It is clear that the globally Lipschitz continuity implies that $|F(y)|\leq L(|y|+1)\quad \forall y\in \re^3$, but in the following proof we want to underline separately the
dependence on the Lipschitz continuity and the linear behaviour at infinity.
\begin{lemma}
Under assumptions \eqref{L} there exists a constant $C$ such that
\begin{equation}\label{sublindelta}
|u_\delta(y)|\leq C\bigg(|y|+ \frac{1}{\delta}\bigg), \quad y\in\re^3.
\end{equation}
\end{lemma}
\begin{proof}
The comparison principle for equation (\ref{probcellappross0}) comes from Da Lio - Ley (\cite{DLL}) 
(this is true also for elliptic operators, see \cite{BCM}).
The existence of a continuous viscosity solution $u_\delta$ in $\re^3$ comes from Perron Theorem, by finding sub- and super solution in $\re^3$. We remark that, in the case when $F$ is bounded by a constant $C_F$ in $\re^3$, a trivial supersolution is $ \frac{C_F}{\delta}$ (a subsolution  $-\frac{C_F}{\delta}$) and the result easily follows.\\
In the more general, sublinear case, let us introduce some constants:
\begin{eqnarray}
&&l:=\min\{ k_1-4, k_2-4, k_3 \} >0   \label{l}\\
&&C_l:=\max\{1,\frac{2 C_F}{l}\}  \label{Cl}\\
&&r_0:=1+\frac{2 C_l}{C_F} \label{r0}.
%.\nn
\end{eqnarray}
Let $U_0$ any regular function $U_0\in C^2(\re^3)$ such that $U_0(y)=\vert y \vert +1$ in $\re^3\setminus B_{r_0}$.
(For instance $U_0(r)=C_0+C_2r^2+C_4r^4$ in $\overline{B(0,r_0)}$ and $U_0(r)=r +1$ in $\re^3\setminus B_{r_0}$).\\
There exists a constant $M_0$ independent of $\delta\in[0,1]$ such that
 \begin{equation}
\label{M0}
\delta U_0-{\cal L}U_0-F(y)\geq -M_0 \quad \hbox{ in } \, B_{r_0}
\end{equation}
Let us define the function $U\in C^2(\re^3)$ as
\[
U(\cdot)=C_l \left(U_0(\cdot)+\frac{2M_0}{\delta}\right).
\]
We claim that $U$ is a supersolution to~\eqref{probcellappross0}.
Let us test the supersolution property first in $B_{r_0}$ then in $\re^3\setminus B_{r_0}$.
For $y\in B_{r_0}$, thanks to $C_l\geq 1$ (see \eqref{Cl}) 
and the definition of $M_0$ in \eqref{M0}, we have:
\begin{eqnarray*}
\delta U-{\cal L}U-F(y)&=&C_l(\delta U_0+2M_0-{\cal L}U_0)-F(y)\\
&\geq&
(C_l-1)F(y)+C_l M_0 \geq 0
\end{eqnarray*}
if $M_0$ is sufficiently large.
For $y\in \re^3\setminus B_{r_0}$, denote $r=\vert y \vert$
\begin{eqnarray*}
&&\delta U(y)-{\cal L}(y, DU(y),D^2U(y))-F(y) \\
&&\quad=C_l\left(\delta(|y|+1+\frac{2M_0}{\delta})
-\frac {2+4(y_1^2+y_2^2)}{r}
+\frac{(k_1y_1^2+k_2y_2^2+k_3y_3^2)}{r}\right)+\\
&&\qquad+C_l\frac {1}{r^3}(y_1^2+y_2^2)(1+4y_3^2)-F(y)\\
&&\quad \geq C_l\left(-\frac {2}{r}+\frac{(k_1-4)y_1^2+(k_2-4)y_2^2+k_3y_3^2)}{r}\right)-F(y)\\
&&\quad\geq C_l(-\frac {2}{r} +l r)-C_F(r+1)\\
&&\quad\geq 0
\end{eqnarray*}
where we used \eqref{l}, \eqref{Cl} and \eqref{r0}.
From the comparison principle then $u_\delta\leq U(y)=C_l \left(U_0+\frac{2M_0}{\delta}\right) \leq C(\vert y \vert +\frac{1}{\delta})$.
The same method applies to define a subsolution and to prove that $u_\delta\geq - C(\vert y \vert +\frac{1}{\delta})$.
\end{proof}

\begin{theorem}\label{lipschitzNuovo}
Under assumptions \eqref{L},
let $u_\delta$ be the unique continuous solution of (\ref{probcellappross0}) which satisfies \eqref{sublindelta}. 
 there holds
$$|u_\delta(y')-u_\delta(y)|\leq \psi(|y'-y|) \qquad \forall  y,y'\in\re^3,$$
where $\psi\in C^2(\re)$ is a concave increasing function with $\psi(0)=0$ and 
$\psi^{\prime}>\max\{\frac{L}{k_1-4}, \frac{L}{k_2-4}, \frac{L}{k_3}\}$ (recall that $L$ is the Lipschitz constant of~$F$) and it is independent of $\delta$.
In particular there holds
\begin{equation}
\label{Lipud}
|u_\delta(y')-u_\delta(y)|\leq \overline L(|y'-y|) \qquad \forall  y,y'\in\re^3, \forall \delta>0
\end{equation}
for $\overline L>\frac{L}{l}$, where $l$ is defined in \eqref{l}.
\end{theorem}

\begin{proof}
For each $\eta>0$, we introduce the function
$$\Psi(x,y)=u(x)-u(y)-\psi(|x-y|)-\eta|x|^2-\eta|y|^2$$
where $\psi$ is a function as in the statement and for simplicity we take $u:=u_\delta$.

Assume for the moment that there holds
\begin{equation}\label{claim}
\Psi(x,y)\leq \frac{8\eta}{\delta}\qquad \forall x,y \in\re^3,\, \eta\in(0,1);
\end{equation}
then, for any $x, y\in\re^3$, as $\eta\to0^+$, we obtain the following inequality
\begin{equation*}
u(x)-u(y)\leq \psi(|x-y|)
\end{equation*}
which is equivalent to the statement because of the arbitrariness of~$x$ and~$y$.

Let us now prove inequality~\eqref{claim}; to this end, we shall proceed by contradiction. Let $(\overline x, \overline y)$ be a maximum point of function $\Psi$ in $\re^3\times \re^3$.
This maximum does exist since from (\ref{sublindelta})
we have that $\lim_{x\rightarrow +\infty}\frac{u(x)}{|x|^2}=0$.

Let us assume by contradiction that 
\begin{equation}\label{contradd}
\Psi(\overline x, \overline y)=u(\overline x)-u(\overline y)-\psi(|\overline x-\overline y|)-\eta|\overline x|^2-\eta|\overline y|^2> \frac{8\eta}{\delta}.
\end{equation}
Clearly, the points $\overline x$ and $\overline y$ cannot coincide, otherwise \eqref{contradd} is false.
We set $\tilde \psi(x,y):=\psi(|x-y|)+\eta(|x|^2+|y|^2)$ and we invoke \cite[Theorem 3.2]{CIL}: for every $\rho>0$ there exist two symmetric $3\times 3$ matrices $X$ and $Y$ such that
\begin{eqnarray}
&\label{cil1}
\left(p_x, X\right)\in {\mathcal J}^{2,+}u(\overline x),\quad
\left(p_y,Y\right)\in {\mathcal J}^{2,-}u(\overline y),\\
&\left(\begin{array}{cc} X&0\\0&-Y\end{array}\right)\leq
A+\rho A^2,  \label{cil3}
\end{eqnarray}
where
\begin{equation*}
p_x:=D_x\tilde \psi(\overline x,\overline y),\quad p_y:=-D_y\tilde \psi(\overline x,\overline y),\quad
A:=\left(\begin{array}{cc} D^2_{xx}\tilde \psi(\overline x,\overline y)
&D^2_{xy}\tilde \psi(\overline x,\overline y)\\D^2_{yx}\tilde \psi(\overline x,\overline y)&D^2_{yy}\tilde \psi(\overline x,\overline y)\end{array}\right).
\end{equation*}
We write explicitly $p_x$, $p_y$ and $A$:
\begin{eqnarray}\label{px}
p_x&=& \psi^{\prime}(|\overline x-\overline y|)\frac{\overline x-\overline y}{|\overline x-\overline y|}+2\eta\overline x= \psi^{\prime}(|\overline x-\overline y|) q +2\eta\overline x\\ \label{py}
p_y &=& \psi^{\prime}(|\overline x-\overline y|)\frac{\overline x-\overline y}{|\overline x-\overline y|}-2\eta\overline y= \psi^{\prime}(|\overline x-\overline y|) q -2\eta\overline y
\end{eqnarray}
where we defined 
\begin{equation}\label{q}
q:=\frac{\overline x-\overline y}{|\overline x-\overline y|}.
\end{equation}
Defining $B:=\frac{I-q\otimes q}{|\overline x-\overline y|}$, 
the matrix $A$ assumes the following form
\begin{multline}\label{A}
A=\psi^{\prime}(|\overline x-\overline y|)\left(\begin{array}{cc} B&-B\\-B&B\end{array}\right)+ \psi^{\prime\prime}(|\overline x-\overline y|)\left(\begin{array}{cc} q\otimes q&-q\otimes q\\-q\otimes q&q\otimes q\end{array}\right)\\+2\eta\left(\begin{array}{cc} I&0\\0&I\end{array}\right).
\end{multline}

From the definition of sub and supersolution and $(p_x, X)$, $(p_y, Y)$, we have
\begin{eqnarray*}
\delta u(\overline x)-tr(\sigma(\overline x)\sigma^T(\overline x)X)-b(\overline x)p_x&\leq& F(\overline x),\\
\delta u(\overline y)-tr(\sigma(\overline y)\sigma^T(\overline y)Y)-b(\overline y)p_y&\geq& F(\overline y).
\end{eqnarray*}
Subtracting the latter inequality from the former, we infer
\begin{multline}\label{totale}
\delta (u(\overline x)-u(\overline y))-
tr\bigg(\sigma(\overline x)\sigma^T(\overline x)X-\sigma(\overline y)\sigma^T(\overline y)Y\bigg)\\+\bigg(-b(\overline x)p_x+b(\overline y)p_y\bigg)\leq F(\overline x)-F(\overline y).
\end{multline}
We want to estimate from below the three terms on the left hand side of \eqref{totale}:
\begin{itemize}
\item[$i$)] ${\mathcal U}:=\delta (u(\overline x)-u(\overline y))$,
\item[$ii$)] ${\mathcal T}:-tr\bigg(\sigma(\overline x)\sigma^T(\overline x)X-\sigma(\overline y)\sigma^T(\overline y)Y\bigg)$,
\item[$iii$)] ${\mathcal G}:=-b(\overline x)p_x+b(\overline y)p_y$.
\end{itemize}
\noindent ($i$). The assumption by contradiction~\eqref{contradd} yields
\begin{equation}\label{U}
{\mathcal U}:=\delta (u(\overline x)-u(\overline y))\geq \delta \psi(|\overline x-\overline y|)+\delta(\eta|\overline x|^2+\eta|\overline y|^2)+8\eta\geq 8\eta.
\end{equation}

\noindent($ii$). Multiplying relation \eqref{cil3} by $(\zeta, \xi)$ where $\zeta$ and $\xi$ are vectors in $\re^3$ we obtain
$(\zeta, \xi)\left(\begin{array}{cc} X&0\\0&-Y\end{array}\right)(\zeta, \xi)^T\leq
(\zeta, \xi)A(\zeta, \xi)^T+\rho (\zeta, \xi)A^2(\zeta, \xi)^T$.
Then, using \eqref{A}, we have
\begin{multline}\label{1}
\zeta X\zeta^T -\xi Y\xi^T \leq \psi^{\prime}(|\overline x-\overline y|)
\bigg(\zeta B\zeta^T -\xi B\zeta^T-\zeta B\xi^T +\xi B\xi^T\bigg)+\\
+\psi^{\prime\prime}(|\overline x-\overline y|)(<\zeta-\xi, q>)^2+
2\eta(|\zeta|^2+|\xi|^2)+ \rho a(\zeta,\xi)
\end{multline}
where we denoted by 
$a(\zeta,\xi):=(\zeta, \xi)A^2(\zeta, \xi)^T$ and $q$ is defined in~\eqref{q}.\\
Recall that, for any choice of two orthonormal basis $\{e_i\}_{i=1,2}$ and $\{\tilde e_i\}_{i=1,2}$ in $\re^2$, 
(if $e_i$ is a orthonormal basis, $\ds tr M=\sum_{i=1}^2 e_i M e_i^T$)
we have
\begin{eqnarray*}
tr(\sigma\sigma^TX)&=&tr(\sigma^TX\sigma)=\sum_{i=1}^2 e_i\sigma^T X \sigma e_i^T\\
tr(\sigma\sigma^TY)&=& tr(\sigma^TY\sigma)= \sum_{i=1}^2 {\tilde e}_i\sigma^T X \sigma {\tilde e}_i^T.
\end{eqnarray*}
We choose
\begin{equation}\label{zixi}
\zeta _i=e_i\sigma^T(\overline x),\quad \xi _i={\tilde e}_i\sigma^T(\overline y);
\end{equation}
 hence, we get
$$
{\mathcal T}=
-\sum_{i=1}^2 e_i\sigma^T(\overline x) X \sigma(\overline x) e_i^T+
\sum_{i=1}^2 {\tilde e}_i\sigma^T(\overline y) Y \sigma(\overline y) {\tilde e}_i^T=
-(\sum_{i=1}^2 \zeta_i X \zeta_i ^T-\sum_{i=1}^2 \xi_i Y \xi_i^T).
$$
Then from inequality \eqref{1} we obtain
\begin{eqnarray*}%\label{2}
{\mathcal T}&\geq& -\psi^{\prime}(|\overline x-\overline y|)
\sum_{i=1}^2 (\zeta_i-\xi_i)B (\zeta_i-\xi_i)^T-
\psi^{\prime\prime}(|\overline x-\overline y|)\sum_{i=1}^2(<\zeta_i-\xi_i, q>)^2\\&&-2\eta\sum_{i=1}^2(|\zeta_i|^2+|\xi_i|^2)- \rho \sum_{i=1}^2a(\zeta_i,\xi_i).
\end{eqnarray*}
From the definition of the matrix $B$ we have
\begin{eqnarray}\label{T}&&\\\notag
{\mathcal T}&\geq& -\frac{\psi^{\prime}}{|\overline x-\overline y|}
\sum_{i=1}^2 (|\zeta_i-\xi_i|^2)+
(\frac{\psi^{\prime}}{|\overline x-\overline y|}-\psi^{\prime\prime})\sum_{i=1}^2(<\zeta_i-\xi_i, q>)^2\\ \notag
&&\qquad-2\eta\sum_{i=1}^2(|\zeta_i|^2+|\xi_i|^2)-\rho \sum_{i=1}^2a(\zeta_i,\xi_i)\\\notag
&\geq& 
-\frac{\psi^{\prime}}{|\overline x-\overline y|}
\sum_{i=1}^2 (|\zeta_i-\xi_i|)^2-
2\eta\sum_{i=1}^2(|\zeta_i|^2+|\xi_i|^2)-\rho \sum_{i=1}^2a(\zeta_i,\xi_i),
\end{eqnarray}
where the last inequality was obtained taking into account that $\psi$ is increasing and concave, so
$\frac{\psi^{\prime}}{|\overline x-\overline y|}-\psi^{\prime\prime}\geq 0$.
\vskip 0.2 cm
($iii$).
From expressions (\ref{px}) and (\ref{py}) of $p_x$ and $p_y$, we have
\[{\mathcal G}=\bigg(-b(\overline x)+b(\overline y)\bigg){\psi^{\prime}}(|\overline x-\overline y|)q
+2\eta\bigg(-b(\overline x)\overline x-b(\overline y)\overline y\bigg).\]

By our choice of $b$ (recall: $b(x)=(-k_1x_1, -k_2x_2,-k_3x_3)$), ${\mathcal G}$ becomes
\begin{multline}\label{G}
{\mathcal G}=\bigg(k_1(\overline x_1-\overline y_1)^2+k_2(\overline x_2-\overline y_2)^2+k_3(\overline x_3-\overline y_3)^2\bigg)
\frac{{\psi^{\prime}}(|\overline x-\overline y|)}{|\overline x-\overline y|}+\\
2\eta \bigg(k_1(\overline x_1^2+\overline y_1^2)+k_2(\overline x_2^2+\overline y_2^2)+k_3(\overline x_3^2+\overline y_3^2)\bigg).
\end{multline}

Now, replacing inequalities \eqref{U}, \eqref{T} and \eqref{G} in \eqref{totale}, we obtain
\begin{eqnarray*}%\label{totale1}
&L|\overline x-\overline y|\geq F(\overline x)-F(\overline y)\geq 
{\mathcal U}+{\mathcal T}+{\mathcal G}\geq\\ 
&8\eta
-\frac{\psi^{\prime}}{|\overline x-\overline y|}
\sum_{i=1}^2 (|\zeta_i-\xi_i|)^2-
2\eta\sum_{i=1}^2(|\zeta_i|^2+|\xi_i|^2)- \rho \sum_{i=1}^2a(\zeta_i,\xi_i)+\nn\\
&\bigg(k_1(\overline x_1-\overline y_1)^2+k_2(\overline x_2-\overline y_2)^2+k_3(\overline x_3-\overline y_3)^2\bigg)
\frac{\psi^{\prime}}{|\overline x-\overline y|}+\nn\\
&
2\eta \bigg(k_1(\overline x_1^2+\overline y_1^2)+k_2(\overline x_2^2+\overline y_2^2)+k_3(\overline x_3^2+\overline y_3^2)\bigg).\nn
\end{eqnarray*}
Passing to the limit as $\rho \rightarrow 0^+$, 
 we obtain
\begin{eqnarray}\label{totale2}
&L|\overline x-\overline y|\geq \\
&\eta\left(8-
2\sum_{i=1}^2(|\zeta_i|^2+|\xi_i|^2) +2k_1(\overline x_1^2+\overline y_1^2)+2k_2(\overline x_2^2+\overline y_2^2)+2k_3(\overline x_3^2+\overline y_3^2)\right)+\nn\\
&+ \frac{\psi^{\prime}}{|\overline x-\overline y|}
\left[-\sum_{i=1}^2 (|\zeta_i-\xi_i|)^2+k_1(\overline x_1-\overline y_1)^2+k_2(\overline x_2-\overline y_2)^2+k_3(\overline x_3-\overline y_3)^2)\right].\nn
\end{eqnarray}
The contradiction is easily obtained choosing as the two orthonormal basis the canonical basis
in $\re^2$, $e_1={\tilde e}_1=(1,0)$ and $e_2={\tilde e}_2=(0,1)$.
Then the vectors  $\zeta_i$ and $\xi_i$ (see \eqref{zixi}), with $\overline x=(\overline x_1,\overline x_2,\overline x_3)$ and 
$\overline y=(\overline y_1,\overline y_2,\overline y_3)$
become
$\zeta_1=(1, 0, 2\overline x_2)$,
$\zeta_2=(0,1, -2\overline x_1)$,
$\xi_1=(1, 0, 2\overline y_2)$,
$\xi_2=(0,1, -2\overline y_1)$, and 
$$|\zeta_1|^2=1+4\overline x_2^2,\ |\zeta_2|^2=1+4\overline x_1^2,\ |\xi_1|^2=1+4\overline y_2^2,\ |\xi_2|^2=1+4\overline y_1^2,$$
 $$|\zeta_1-\xi_1|^2= 4(\overline x_2-\overline y_2)^2,\ |\zeta_2-\xi_2|^2= 4(\overline x_1- \overline y_1)^2.$$
Hence, relation \eqref{totale2} becomes
\begin{eqnarray*}%\label{totale3}
&&L|\overline x-\overline y|\geq\eta[8-8-8(\overline x_1^2+\overline y_1^2+\overline x_2^2+\overline y_2^2) +\\
&&
2k_1(\overline x_1^2+\overline y_1^2)+2k_2(\overline x_2^2+\overline y_2^2)+2k_3(\overline x_3^2+\overline y_3^2)]+\nn\\
&& \frac{\psi^{\prime}}{|\overline x-\overline y|}
[(k_1-4)(\overline x_1-\overline y_1)^2+(k_2-4)(\overline x_2-\overline y_2)^2)+k_3(\overline x_3-\overline y_3)^2]\nn
\\ 
&&\geq
2\eta[(k_1-4)(\overline x_1^2+\overline y_1^2)+
(k_2-4)(\overline x_2^2+\overline y_2^2)+ k_3(\overline x_3^2+\overline y_3^2)]
+\nn\\
&&\frac{\psi^{\prime}}{|\overline x-\overline y|}
[(k_1-4)(\overline x_1-\overline y_1)^2+
(k_2-4)(\overline x_2-\overline y_2)^2+ k_3(\overline x_3-\overline y_3)^2].\nn
\end{eqnarray*}
By our choice of $k_1$, $k_2$ and $k_3$ in ($A2$) (namely, $k_1, k_2>4$ , $k_3>0$) we get
\begin{equation*}
L|\overline x-\overline y|^2\geq 
\psi^{\prime}(|\overline x-\overline y|)
[(k_1-4)(\overline x_1-\overline y_1)^2+
(k_2-4)(\overline x_2-\overline y_2)^2+ k_3(\overline x_3-\overline y_3)^2],\nn
\end{equation*}
thus we obtain a contradiction provided that we choose a function $\psi$ such that
$$\psi^{\prime}>\max\{\frac{L}{k_1-4}, \frac{L}{k_2-4}, \frac{L}{k_3}\}.$$
Hence, the proof of our claim~\eqref{claim} is accomplished.
The second statement of the theorem easily follows by taking $\psi(z)=\overline Lz$, with 
$\overline L>\max\{\frac{L}{k_1-4}, \frac{L}{k_2-4}, \frac{L}{k_3}\}$.
\end{proof}
\begin{remark}
Similar arguments can be applied to other matrices still related to degenerate elliptic operators as, for example, in dimension 2:
 $$ \sigma(y):=(\sigma_{ij}(y))_{i,j} \hbox{ with } \sigma_{ij}(y)= a_{ij}y_1+b_{ij}y_2+c_{ij}$$ 
which in particular encompasses the Ornstein-Uhlenbeck operator  and Grushin operator, respectively
$$
\ds\sigma_{OU}=\left(\begin{array}{cc} 1&0\\0&0\end{array}\right), \qquad \ds\sigma_G=\left(\begin{array}{cc} 1&0\\0&y\end{array}\right).$$
For the Grushin operator in a forthcoming paper \cite{MMT2} we will obtain a local H\"older continuity uniform in $\delta$ using a technique introduced in \cite{FIL}.
\end{remark}

\subsection{A key estimate on the growth of the approximate corrector}

The aim of this section is to establish that the solution to the approximating cell problem has a logarithmic growth at infinity. Our arguments are borrowed form \cite[Proposition 3.2]{G2}.
\begin{lemma}\label{ghilli233}
Assume ($A1$) and ($A2$). Let $u_\d(y)$ be the solution of equation \eqref{probcellappross0} with~\eqref{F}. 
There exists $C>0$ such that
\begin{equation*}
|u_\d(y)-u_\d(0)|\leq C\left[1+\log((y_1^2+y_2^2)^2+y_3^2+1)\right] \qquad \forall y\in \re^3,\, \d\in(0,1).
\end{equation*}
\end{lemma}
\begin{proof}
We can argue as in \cite[Proposition 3.2]{G2}, replacing its Lemma 3.3 with our Theorem~\ref{lipschitzNuovo}; 
to this end, our first step is to claim that, for $C_1$ and $R$ sufficiently large, the function $g(y):=C_1 \log((y_1^2+y_2^2)^2+y_3^2)$ is a supersolution to \eqref{probcellappross0} in $\re^3\setminus B_R$. Indeed, by equality~\eqref{calL2nd} there holds:
\begin{eqnarray*}
tr(\sigma\sigma^TD^2g(y))
&=&\frac{8C_1(y_1^2+y_2^2)}{(y_1^2+y_2^2)^2+y_3^2}\\
b(y)\cdot Dg(y)&=&-C_1\frac{4(y_1^2+y_2^2)(k_1y_1^2+k_2y_2^2)+2k_3y_3^2}{(y_1^2+y_2^2)^2+y_3^2}.
\end{eqnarray*}
By these identities, we get
\begin{equation*}
\delta g(y)-tr(\sigma(y)\sigma^T(y)D^2 g)-b(y)Dg\geq F(y),\ y\in\re^3\setminus B_R,
\end{equation*}
provided that $C$ and $R$ are sufficiently large. 
Now
if $\max_{\overline B_R}u_{\delta}\leq 0$ then we have 
$\max_{\overline B_R}u_{\delta}\leq g(y)$ for any $y\in\partial \overline B_{R}$.
By the comparison principle established in \cite{DLL}, we obtain
$u_{\delta}\leq g$ in $\re^3$.
If $\max_{\overline B_R}u_{\delta}> 0$, we note that 
$g_1(\cdot):=g(\cdot)+\max_{\overline B_R}u_{\delta}$ is still a supersolution of 
\eqref{probcellappross0} in $\re^3\setminus B_R$. Hence, again by the comparison principle we have
$u_{\delta}\leq g_1$ in $\re^3$. By Theorem \ref{lipschitzNuovo} we infer:
$u_\d(y)-u_\d(0)\leq g_1(y)+\overline L R$
which gives one of the two inequalities of the statement. The proof of the other one is similar and we shall omit it.
\end{proof}
\subsection{The cell problem}

\begin{theorem}\label{TH33}
Under assumptions A1)-A6) of Section \ref{pbsp}, for every $(x,p,X)$ the constant $\lambda=-\int_{\re^3} H(\overline x,y,\overline p, \overline X,0)d\mu(y)$ ($\mu$ is the invariant measure defined in Lemma \ref{eqmu}) is the unique constant  such that the cell problem \eqref{cell} admits a solution $w(y)$ which  is globally Lipschitz continuous and satisfies the following estimate:
\begin{equation}\label{wlog}
|w(y)-w(0)|\leq C\left[1+\log((y_1^2+y_2^2)^2+y_3^2+1)\right] \qquad \forall y\in \re^3.
\end{equation}
Moreover, the solution~$w$ is unique up to an additive constant.
\end{theorem}
\begin{proof}
To prove the existence of such a $\lambda$
we argue as in \cite[Proposition 3.2]{G2}, replacing its Lemma 3.3 with 
our Theorem~\ref{lipschitzNuovo}.
We consider the solution $u_{\delta}$ of the approximated cell problem \eqref{probcellappross0}, recalling that, from A3), $F(y)$ is bounded in $\re^3$; then $w_{\delta}(y):=u_{\delta}(y)-u_{\delta}(0)$
satisfies
$$
\delta w_{\delta}(y)-tr(\sigma(y)\sigma^T(y)D^2 w_{\delta})-b(y)Dw_{\delta}=F(y)-\delta u_{\delta}(0).
$$
From the Lipschitz continuity of $u_{\delta}(y)$ in \eqref{Lipud}
we have that 
$$|w_{\delta}(y)|=|u_{\delta}(y)-u_{\delta}(0)|\leq \bar L\vert y \vert$$
and 
$$|w_{\delta}(y)-w_{\delta}(z)|=|u_{\delta}(y)-u_{\delta}(z)|\leq \bar L|y-z|$$
hence $w_{\delta}(y)$ are locally equibounded and equicontinuous. Then by Ascoli-Arzela theorem and standard diagonal argument we can conclude that there exists a function $w$ with the desired properties.
Moreover from Theorem \ref{conv1} we know that
$$\delta u_{\delta} \rightarrow \int H(x,y,p,X,0)d\mu(y)=-\lambda.$$
To prove the uniqueness of $\lambda$ and the uniqueness up to a constant of $w$, we use the arguments of \cite[Thm.4.5]{FIL}. For the sake of completeness, let us recall them briefly.\\
First of all we assume that any solution $w$ of \eqref{cell} is regular and this is not retrictive because the smoothness will be proved in Theorem \ref{regw} in the next section. \\
%We remark that this smoothness depends only on the only result contained in this paper which will be used is the Lipschitz continuity of $w$ (see Theorem \ref{lipschitzNuovo}).\\
Assume by contradiction that there exist two constants $\lambda_1\ne\lambda_2$ and two regular functions $w_1$, $w_2$ such that $(\lambda_1, w_1)$ and $(\lambda_2,w_2)$ are both solutions to problem~\eqref{cell}.
Without any loss of generality we assume $\lambda_1<\lambda_2$. 
%By standard arguments, $w_1$ and $w_2$ are regular solutions. 
We set $u(\cdot):= w_1(\cdot)-w_2(\cdot)$ and $U_1(y):=y_1^4+y_2^4+y_3^2$. Without any loss of generality (eventually adding a constant), we assume $\sup_{\re^3}u>0$. We observe that, for $\gamma>0$ sufficiently small and $\beta>0$ sufficiently large, there hold
\begin{eqnarray}\label{celli}
-tr(\sigma\sigma^TD^2 u)-bDu&=&\lambda_1-\lambda_2\qquad \textrm{in }\re^3\\
-tr(\sigma\sigma^TD^2 U_1)-bDU_1&\geq&\gamma U_1-\beta  \qquad \textrm{in }\re^3.\label{celliU}
\end{eqnarray}
(For instance an explicit and tedious calculation gives: $\gamma<2k_3$, $\gamma<4k_1$, $\gamma<4k_2$, $\beta\geq \frac{100}{4k_1-\gamma}+ \frac{100}{4k_2-\gamma}$.)\\
Consider $\rho>0$ so small to have $\rho \beta<\lambda_2-\lambda_1$
and set $U(\cdot):=\rho U_1(\cdot)$. By the global Lipschitz continuity of $w_1$ and $w_2$, we have 
\begin{equation}
\label{FIL2}
\lim_{|y|\to \infty}(u(y)-U(y))=-\infty.
\end{equation}
Hence, there exists an open bounded set $\Omega\subset \re^3$ such that $u\leq U$ in $\re^3\setminus \Omega$.
By linearity of the operator, relations \eqref{celli} and \eqref{celliU} entail that the function $\eta(\cdot):=u(\cdot)-U(\cdot)$ satisfies
\[-tr(\sigma\sigma^TD^2 \eta)-bD\eta\leq -\rho \gamma U_1 -(\lambda_2-\lambda_1-\rho \beta)<0\qquad \textrm{in }\re^3\]
where the last inequality is due to our choice of $\rho$ and to $U_1\geq 0$.
Applying the maximum principle to $\eta$ on the domain $\Omega$ we obtain: $\eta\leq 0$ in $\Omega$. Hence, we have: $\eta\leq 0$ in $\re^3$, namely
\[u(y)\leq \rho U_1(y)\qquad \forall y\in\re^3.\]
Letting $\rho\to0^+$, we get $u\leq 0$ in $\re^3$ which gives the desired contradiction.

Let us now pass to prove that if $(\lambda,w_1)$ and $(\lambda,w_2)$ are both solutions to \eqref{cell} then $w_1=w_2+C$, for some constant $C$. By \eqref{celliU}, there exists $R>0$ such that
\[-tr(\sigma\sigma^TD^2 U_1)-bDU_1> 0\qquad \textrm{for }|y|>R.\]
For $u=w_1-w_2$ as before, we claim that there holds
\begin{equation}\label{claimFIL}
\sup_{\re^3}u=\max_{\overline{B(0,R)}}u.
\end{equation}
Actually, for any $\rho>0$, for $\eta(\cdot)=u(\cdot)-\rho U_1(\cdot)$,
\[-tr(\sigma\sigma^TD^2 \eta)-bD\eta< 0 \qquad \textrm{for }|y|>R.\]
As before (see \eqref{FIL2}) $\ds \lim_{|y|\to\infty}\eta(y)=-\infty$ and this implies that $\eta$ attains his maximum on $\re^3$.\\
By the maximum principle, $\eta$ cannot attain its maximum over $\overline{B(0,R)}^C$ at any point in its interior. Then
\[u(y)-\rho U_1(y)\leq \max_{|y'|=R}\left(u(y')-\rho U_1(y')\right)\qquad \forall |y|\geq R;\]
letting $\rho\to0^+$, we obtain our claim \eqref{claimFIL}.

By \eqref{claimFIL}, for any $r>R$, the strong maximum principle on $u$ over $B(0,r)$  ensures that $u$ is a constant function on $B(0,r)$. By the arbitrariness of $r$, we obtain the desired result. 
\end{proof}

\subsection{Regularity of the corrector $w$}
In this subsection we prove the $C^2$-regularity of the corrector $w$.  This result seems not straightforward. Actually, since our operator $\cal L$ contains second order horizontal derivatives and Euclidean first derivatives as well, the second order part of the operator does not immediately regularize  the first order one.
On the other hand is worth to observe that, for $H(x,\cdot,p,X,0)\in C^{\infty}$, the solution $w$ of \eqref{cell} is $C^{\infty}$ by hypoellipticity.

We start with a lemma which states the equivalence between solution in the sense of distributions and continuous viscosity solutions, under a growth condition at infinity.

\begin{lemma}\label{ishii}
Consider the equation 
\begin{equation}\label{eqlemma}
-tr(\sigma\sigma^TD^2 \chi)-b(y)D\chi +K\chi=R(y),\ y\in\re^3,
\end{equation}
where $R$ is a bounded globally Lipschitz continuous function and $K$ is a strictly positive constant.
Then\\
1) there exists a unique bounded and globally Lipschitz continuous viscosity solution $\chi$;\\
2) $\chi$ is a solution in the sense of distributions;\\
3) any  bounded solution of \eqref{eqlemma} in the sense of distributions coincides with $\chi$.
\end{lemma}
\begin{proof}
1)  Follows from Lemma \ref{exudelta0} and Theorem \ref{lipschitzNuovo}. \\
2) Follows from
\cite[Theorem 1]{Ishii}.\\
3) If $\chi_1$ is a solution of \eqref{eqlemma} in the sense of distributions, we define $\overline \chi:=\chi_1-\chi$, ($\overline\chi$ is bounded). By linearity, $\overline\chi$ solves in the sense of distributions
$$-tr(\sigma\sigma^TD^2 \overline\chi)-b D\overline\chi +K\overline\chi=0.$$
By the hypoellipticity of the operator, $\overline\chi$ is smooth. Hence 
$\chi_1=\overline \chi+\chi$ is continuous and, by \cite[Theorem 2]{Ishii} is also a viscosity solution. By the uniqueness of sublinear viscosity solutions of \eqref{eqlemma}, we get 
$\chi_1=\chi$.
\end{proof}

\begin{theorem}\label{regw}
Under assumptions (A1)-(A6),
let $w$ be the solution of the cell problem \eqref{cell} found in Theorem \ref{TH33}. Then $w\in C^{2,\alpha}_{loc}(\re^3)$, for some $\alpha\in(0,1)$.
\end{theorem}
\begin{proof}
Let us denote by $X_1=(\partial_{y_1}+2y_2\partial_{y_3})$ and $X_2=(\partial_{y_2}-2y_1\partial_{y_3})$ the two vector fields associated to the two columns of the matrix $\sigma$.
Recall that these two vectors are the generators of the Heisenberg group and span all $\re^3$ because their commutator is $[X_1, X_2]=-4\frac{\partial}{\partial y_3}$.\\
Along the proof $\alpha$ is a strictly positive constant which may change line to line.\\
The corrector $w$ solves 
\begin{equation}\label{eqconG}
-tr(\sigma(y)\sigma^T(y)D^2 w(y))-b(y)Dw(y)=G(y)
\end{equation}
with $$
G(y):=\lambda-H(\overline x,y,D_x\psi(\overline t, \overline x),D_{xx}^2\psi(\overline t, \overline x),0).
$$
First let us get the Lipschitz continuity of $\frac{\partial w}{\partial y_3}$.
Derivating equation \eqref{eqconG} with respect to $y_3$ we obtain that the function $u:=\frac{\partial w}{\partial y_3}$ solves in the sense of distributions
 \begin{equation}\label{equ}
-tr(\sigma\sigma^TD^2 u)-bDu+k_3u=\frac{\partial G}{\partial y_3},
\end{equation}
with $k_3>0$ by assumption A2). Note that $u$ is bounded by Theorem \ref{lipschitzNuovo}, then by Lemma \ref{ishii} we get that $u$ is Lipschitz continuous and it is also a viscosity solution; hence 
\begin{equation}\label{Kri0}
u=\frac{\partial w}{\partial y_3}\in BLip(\re^3).
\end{equation}
Deriving equation \eqref{equ} with respect to $y_3$, we get that the function 
$z:={\partial^2 w}/{\partial y_3^2}$ solves 
\begin{equation}\label{eqz}
-tr(\sigma\sigma^TD^2 z)-bDz+2k_3z=\frac{\partial^2 G}{\partial y_3^2}.
\end{equation}
By assumption A6), we can apply Lemma \ref{ishii} also to~\eqref{eqz} and we get that the function $z$ is globally Lipschitz continuous, i.e.
\begin{equation}\label{w33c2}
\frac{\partial^2 w}{\partial y_3^2}\in BLip(\re^3).
\end{equation}
Now we study the regularity of $w$ with respect to $y_1$ and $y_2$; to this end, let us come back to \eqref{eqconG}. 
From the Lipschitz continuity of $w$ (see Theorem~\ref{lipschitzNuovo}), we get 
 \begin{equation*}\label{11}
-tr(\sigma\sigma^TD^2 w)\in L^{\infty}_{loc}(\re^3).
\end{equation*}
By \eqref{Kri0}, we have
\begin{equation*}%\label{uffa21}
\frac{\partial^2 w}{\partial y_1\partial y_3}, \frac{\partial^2 w}{\partial y_2\partial y_3},
\frac{\partial^2 w}{\partial y_3^2}\in L^\infty(\re^3).
\end{equation*}
Taking into account the explicit expression of $-tr(\sigma\sigma^TD^2 w)$
we have that 
\begin{equation*}%\label{uffa31}
\frac{\partial^2 w}{\partial y_1^2}+\frac{\partial^2 w}{\partial y_2^2} \in L^{\infty}_{loc}(\re^3).
\end{equation*}
This relation and \eqref{w33c2} imply $\Delta w\in L^\infty_{loc}(\re^3)$, ($\Delta$ is the Euclidean Laplacian).
Hence from classical results on uniformly elliptic equations we obtain
\begin{equation*}%\label{gradientecontinuo}
Dw\in C^{0,\alpha}_{loc}(\re^3).
\end{equation*}
Now we can replace $w$ with $u$,  \eqref{Kri0} with \eqref{w33c2} and \eqref{eqconG} with \eqref{equ}, using the same arguments we get:
\begin{equation}\label{gradientecontinuo}
Du\in C^{0,\alpha}_{loc}(\re^3).
\end{equation}

In particular
\begin{equation*}%\label{uffa2}
\frac{\partial^2 w}{\partial y_1\partial y_3}, \frac{\partial^2 w}{\partial y_2\partial y_3},
\frac{\partial^2 w}{\partial y_3^2}\in C^{0,\alpha}_{loc}(\re^3).
\end{equation*}
As before we have that $\Delta w\in C^{0,\alpha}_{loc}(\re^3)$.
Hence, from classical results on uniformly elliptic equations we obtain the statement.

\end{proof}
\begin{remark}{\rm
We remark that in this proof 
 the structure of the operator ${\cal L}$ and Theorem~\ref{lipschitzNuovo} play a crucial role, this allows us to overcome the application of some deep results on the hypoelliptic theory.}
\end{remark}

\section{Proof of Theorem \ref{maintheor}.}\label{sect:proof}

In this section we prove 
the convergence of the solution $V^{\epsilon}$ of \eqref{HJBSP} to the solution of the effective equation \eqref{EFFSP}.

 Let us recall from Proposition~\ref{Existence} that, for every compact $K\subset  \re^n$, the solutions $V^\epsilon$ are equibounded in $(0,T)\times K\times \re^3$, hence the following semilimits
\begin{eqnarray*}
\overline V(t,x,y)&=&\limsup_{\epsilon\to 0^+, t'\to t, x'\to x, y'\to y}V^{\epsilon}(t',x',y')\qquad \textrm{for } t<T\\
\overline V(T,x,y)&=&\limsup_{t'\to T^-, x'\to x, y'\to y}\overline V(t',x',y')\qquad \textrm{for } t=T;
\end{eqnarray*}
(and similarly for $\underline V$ replacing $\limsup$ by $\liminf$)  are well defined.
This two step definition of $\overline V$ is needed to overcome an expected initial layer.

For the sake of clarity we shall divide the proof in several steps, as follows:
\begin{itemize}
\item[Step 1.] $\overline V$ and $\underline V$ are independent of $y$;
\item[Step 2.] $\overline V$ and $\underline V$ are respectively a subsolution and a supersolution of the parabolic equation \eqref{EFFSP};
\item[Step 3.] $\overline V(T,x)\leq \overline g(x)\leq \underline V(T,x)$, where $\overline g(x)$ is defined in  \eqref{opeff};
\item[Step 4.]  $\overline V= \underline V=:V$ and $V^{\epsilon}\rightarrow V$ locally uniformly.
\end{itemize}

\subsection{Step 1}
\begin{lemma}\label{Vindy}
Under assumptions A1)-A5), $\overline V$ and $\underline V$ are independent of $y$.
\end{lemma}
\begin{proof}
Let us observe that $\overline V(t,x,y)$ and $\underline V(t,x,y)$ are respectively 
BUSC and BLSC. We prove that $\overline V(t,x,y)$ is independent of $y$; being similar, the proof for $\underline V(t,x,y)$ is omitted.

We claim that for $(t_0,x_0)\in(0,T)\times \re^n$ fixed, $\overline V (t_0,x_0, y)$ is a subsolution for $y\in\re^3$
to equation 
\begin{equation}\label{liouvi}
-tr(\sigma\sigma^T(y) D_{yy}V)-b(y)\cdot D_yV = 0.
\end{equation}
Assuming for the moment that is true, since $\overline V(t_0,x_0,y)$ is BUSC in $y$,
we can apply the Liouville theorem proved in \cite[Proposition 3.1]{MMT} to deduce that the function $\overline V(t_0,x_0,y)$ does not depend on $y$.

In order to prove that $\overline V (t_0,x_0, y)$ is a subsolution 
to equation \eqref{liouvi} we follow the same arguments as in Step 2 of \cite[Theorem 3.2]{BC}, which we write for the sake of completeness.

First of all we prove that 
$\overline V (t, x, y)$ is a subsolution 
to equation \eqref{liouvi} for $(t,x,y)\in(0,T)\times \re^3\times \re^3$. To do this we fix a point $(\overline t, \overline x, \overline y)$ and a smooth function $\psi$ such that 
$\overline V-\psi$ has a local strict maximum at $(\overline t, \overline x, \overline y)=\overline P$ in
$\overline{B(\overline P, r)}=\{(t, x, y):\vert (t, x, y)-(\overline t, \overline x, \overline y)\vert \leq r\}$, for some $r>0$.
Using the definition of the half relaxed semilimit it is possible to prove (see \cite{BCD}) that there exists 
$\epsilon_n\to 0$ and $(t_n, x_n, y_n)\in \overline{B(\overline P, r)}$ such that  $(t_n, x_n, y_n)\to (\overline t, \overline x, \overline y)$,
$(t_n, x_n, y_n)$ are maxima for $V^{\epsilon_n}-\psi$ in $\overline{B(\overline P, r)}$ and 
$V^{\epsilon_n}(t_n, x_n, y_n)\to \overline V(\overline t, \overline x, \overline y)$.
Since $V^{\epsilon_n}$ solves \eqref{HJBSP} then 
\begin{eqnarray*}
&&-\partial_t\psi+H\left(x_n,y_n,D_x\psi,D_{xx}^2\psi,\frac{D_{xy}^2\psi}{\sqrt{\epsilon_n}}\right)\\
&&-\frac1\epsilon_n
tr(\sigma(y_n)\sigma(y_n)^T D_{yy}\psi)-\frac1\epsilon_n b(y_n)\cdot D_y\psi
+a V^{\epsilon_n}(t_n, x_n, y_n)\leq 0
\end{eqnarray*}
Then
\begin{eqnarray*}
&&-tr(\sigma(y_n)\sigma(y_n)^T D_{yy}\psi)-b(y_n)\cdot D_y\psi
\leq \\
&&\epsilon_n
\left [\partial_t\psi-H\left(x_n,y_n,D_x\psi,D_{xx}^2\psi,\frac{D_{xy}^2\psi}{\sqrt\epsilon}\right)-
a V^{\epsilon_n}\right ]
\end{eqnarray*}
From the regularity of $\psi$, the continuity of $H$ (obtained from A1), A3), A5)) and the local uniform boundedness of 
$V^{\epsilon_n}$, the part in the brackets on the right hand side is uniformly bounded with respect to $n$
in $B(\overline P, r)$,  then passing to the limit as $\epsilon_n\to 0$ we get
$$
-tr(\sigma(\overline y)\sigma(\overline y)^T D_{yy}\psi)- b(\overline y)\cdot D_y\psi
\leq 0,
$$
i.e. 
$\overline V (t, x, y)$ is a subsolution 
to equation \eqref{liouvi} for $(t,x,y)\in(0,T)\times \re^n\times \re^3$.

We now show that, arguing as in \cite[Lemma II.5.17]{BCD}, for every fixed $(t_0, x_0)\in(0,T)\times \re^n$ the function $\overline V (t_0, x_0, y)$ is a subsolution to equation \eqref{liouvi}.
We fix $\overline y$ and $\phi(y)$, a smooth function such that $\overline V (t_0, x_0, y)-\phi(y)$ has a strict local maximum at $\overline y$ in $B(\overline y, \delta)$ and such that 
$\phi(y)\geq 1$ in $B(\overline y, \delta)$. Let us chose $\delta>0$ small enough s.t. $t_0-\delta>0$.
We define, for $\eta>0$ 
$$\phi_{\eta}(t,x,y)=\phi(y)\left( 1+\frac{|x-x_0|^2+|t-t_0|^2}{\eta}\right)$$
 and we consider 
$(t_{\eta}, x_{\eta}, y_{\eta})$ a maximum point of $\overline V -\phi_{\eta}$ in $\overline{B((t_0, x_0, \overline y),\delta)}$.
We remark that
\begin{equation}
\label{eq:max1}
\overline V(t_{\eta}, x_{\eta}, y_{\eta})-\phi_{\eta}(t_{\eta}, x_{\eta}, y_{\eta})\geq \overline V(t_0, x_0, \overline y)-\phi(\overline y)
\end{equation}
\begin{equation*}%\label{eq:max2}
\overline V(t_{\eta}, x_{\eta}, y_{\eta})-\phi( y_{\eta})\geq \overline V(t_0, x_0, \overline y)-\phi(\overline y)
\end{equation*}
and we can prove that, eventually passing to subsequences (as $\eta\to 0$) first that $(t_{\eta}, x_{\eta})\to (t_0, x_0)$, then that
$y_{\eta}\to \overline y$ using the strict maximum property of $\overline y$.\\
Using  \eqref{eq:max1} and the upper semicontinuity of $\overline V$
$$K_{\eta}:= \left( 1+\frac{|x_{\eta}-x_0|^2+|t_{\eta}-t_0|^2}{\eta}\right)\to K>0.$$
Now, using the fact that $\overline V(t,x,y)$ is a subsolution of \eqref{liouvi} in $(t_{\eta}, x_{\eta}, y_{\eta})$ we get 
$$
-tr(\sigma(y_{\eta})\sigma(y_{\eta})^T D^2_{yy}\phi_\eta)- b(y_{\eta})\cdot D_y\phi_\eta \leq 0,
$$
which gives, passing to the limit as $\eta\to 0$ 
$$
-tr(\sigma(\overline y)\sigma(\overline y)^T D^2_{yy}\phi)- b(\overline y)\cdot D_y\phi \leq 0.
$$
\end{proof}
\begin{remark}
Using \eqref{eq:max1} and the upper semicontinuity of $\overline V$ it is possible to prove that
$K_{\eta}\to 1$. This property in not used in our proof but can be useful in more general and nonlinear cases.
\end{remark}

\subsection{Step 2}
\begin{proposition}
 Under the assumptions A1)-A5), $\overline V$ and $\underline V$ are respectively a subsolution and a supersolution of the parabolic equation in~\eqref{EFFSP}.
\end{proposition}
\begin{proof}
We prove that $\overline V$ is a viscosity subsolution of \eqref{EFFSP} in $  ]0,T[\times \re^n$. The proof that $\underline V$ is a viscosity supersolution is analogous, so we shall omit it.

We take a smooth test function $\psi(t,x)$ such that $(\overline t, \overline x)\in ]0,T[\times \re^n$ is a strict local maximum point for
$\overline V -\psi$. We have to prove that 
$$
-\partial_t\psi(\overline t, \overline x)+
\overline H\left(\overline x, D_x\psi(\overline t, \overline x), D_{xx}^2\psi(\overline t, \overline x)\right)+a\, \psi (\overline t, \overline x)\leq0.$$
Without any loss of generality we can assume that:
\begin{enumerate}
\item $\overline V(\overline t, \overline x)=\psi(\overline t, \overline x)$;
\item $\psi$ is coercive in $x$ uniformly in t, i.e.
\begin{equation}
\label{cresce3}
 \lim_{\vert x \vert \rightarrow \infty} \inf_{t\in [0,T]}\psi(t,x)=+\infty;
\end{equation}
\item there holds
\begin{equation}
\label{bordoT}
\ds\inf_{x\in \re^n}\psi(\frac{\bar t}{2},x)>M+1, \quad \inf_{x\in \re^n}\psi(\frac{\bar t+T}{2},x)>M+1
\end{equation}
 where $M$ is a constant such that $\vert V^\epsilon(t,x,y)\leq M$;
\item $\ds\sup_{(t,x)\in K}|\partial_t\psi(t,x)|\leq C_K$ for any $K$ compact in $[0,T]\times \re^n$.
\end{enumerate}

For any fixed $\eta\in]0,1]$, let us consider now the "perturbed test function":
$$\psi_{\epsilon\eta}(t,x, y):=\psi(t,x)+\epsilon(w(y)+\eta\chi(y))$$
where $w(y)$ is the viscosity solution of the cell problem \eqref{cell} founded in Theorem \ref{TH33} 
associated to $(\overline x,D_x\psi(\overline t, \overline x), D_{xx}^2\psi(\overline t, \overline x))$
and $\chi(y)$ 
is the Lyapunov function 
\begin{equation}\label{Lyap}
\chi(y)=y_1^2+y_2^2+y_3^2.
\end{equation}
Note that, from \eqref{wlog} and the definition of $\chi$ in \eqref{Lyap}, we have
\begin{equation}\label{cresce}
w(y)+\eta\chi(y)\to +\infty,\ \mbox{if}\ |y|\to +\infty.
\end{equation}
and there exists a constant $k_0$ independent of $\eta$ such that
\begin{equation}\label{wmin}
w(y)+\eta\chi(y)\geq - k_0(1+\log(\eta)).
\end{equation}
Let consider the function 
\[
\Psi(t,x,y):=V^\epsilon(t,x,y)-\psi_{\epsilon\eta}(t,x, y)
\]
Thanks to the equi-boundedness of $V^\epsilon$, \eqref{cresce} and \eqref{cresce3} we have:
\begin{equation*}%\label{cresce2}
\Psi(t,x,y)\to -\infty,\ \mbox{if}\ (x,y)\to +\infty 
\end{equation*}
and there exists a point $\ds(t_{\epsilon, \eta},x_{\epsilon, \eta},y_{\epsilon, \eta})\in [\frac{\bar t}{2},\frac{\bar t+T}{2}]\times \re^n\times \re^3$ which is a global maximum point of $\Psi$  in $\ds [\frac{\bar t}{2},\frac{\bar t+T}{2}]\times \re^n\times \re^3$.

\begin{claim}
\label{claimbddtx}
$\ds(t_{\epsilon, \eta},x_{\epsilon, \eta})$ is bounded uniformly in $\epsilon$.
\end{claim}

The points $t_{\epsilon,\eta}$ are obviously bounded.
Now using the maximum property of $\ds(t_{\epsilon, \eta},x_{\epsilon, \eta},y_{\epsilon, \eta})$, we have:
\[
V^\epsilon(t_{\epsilon, \eta},x_{\epsilon, \eta},y_{\epsilon, \eta})-\psi_{\epsilon\eta}(t_{\epsilon, \eta},x_{\epsilon, \eta},y_{\epsilon, \eta})\geq    V^\epsilon(\bar t, 0, 0)-\psi(\bar t, 0)-\epsilon(w(0)+\eta\chi(0);
\]
then from \eqref{wmin} 
\begin{multline*}%\label{eqbddtx}
K\geq V^\epsilon(t_{\epsilon, \eta},x_{\epsilon, \eta},y_{\epsilon, \eta})-V^\epsilon(\bar t, 0, 0)+\psi(\bar t, 0)+\epsilon(w(0)+\eta\chi(0))\geq   \\ 
\psi(t_{\epsilon, \eta},x_{\epsilon, \eta})+\epsilon(w(y_{\epsilon, \eta})+\eta\chi(y_{\epsilon, \eta}))
\geq \psi(t_{\epsilon, \eta},x_{\epsilon, \eta})-\epsilon k_0(1+\log(\eta))
\end{multline*}
and we end the proof of Claim~\ref{claimbddtx} using and \eqref{cresce3}.

\begin{claim}
\label{claimtbord}
If $\ds t_{\epsilon, \eta}=\frac{\bar t}{2}$ or $\ds t_{\epsilon, \eta}=\frac{\bar t+T}{2}$ , then for any $(t', x', y')\in  [\frac{\bar t}{2},\frac{\bar t+T}{2}]\times \re^n\times \re^3$ 
\[
\Psi(t', x', y')\leq -1+\epsilon k_0(1+\log(\eta)).
\]
\end{claim}

Thanks to  \eqref{wmin}
\[
\Psi(t', x', y')\leq V^\epsilon(t_{\epsilon, \eta},x_{\epsilon, \eta},y_{\epsilon, \eta})-\psi(t_{\epsilon, \eta},x_{\epsilon, \eta})+\epsilon k_0(1+\log(\eta)).
\]
Using now \eqref{bordoT} and the definition of $M$
\begin{eqnarray*}
\Psi(t', x', y')&\leq& V^\epsilon(t_{\epsilon, \eta},x_{\epsilon, \eta},y_{\epsilon, \eta})-(M+1)+\epsilon k_0(1+\log(\eta))\\
&\leq&  -1+\epsilon k_0(1+\log(\eta)).
\end{eqnarray*}

\begin{claim}\label{eqVep}
If $\ds t_{\epsilon, \eta}\in]\frac{\bar t}{2},\frac{\bar t+T}{2}[$ , then
\begin{multline}\label{claim3}
-\partial_t\psi(t_{\epsilon, \eta},x_{\epsilon, \eta})+
\bar H\left(\overline x, D_x\psi(\overline t, \overline x), D_{xx}^2\psi(\overline t, \overline x)\right)\\
-\eta {\cal {L}}(y_{\epsilon, \eta},D_y\chi(y_{\epsilon, \eta}), D^2_{yy}\chi(y_{\epsilon, \eta}))+{\cal{F_{\epsilon, \eta}}}
+a V^\epsilon(t_{\epsilon, \eta},x_{\epsilon\eta},y_{\epsilon, \eta})\leq 0,
\end{multline}
where 
\begin{multline*}
{\cal{F_{\epsilon, \eta}}}=H(x_{\epsilon, \eta},y_{\epsilon, \eta},D_x\psi(t_{\epsilon, \eta},x_{\epsilon, \eta}),D_{xx}^2\psi(t_{\epsilon, \eta},x_{\epsilon, \eta}),0)\\-H(\overline x,y_{\epsilon, \eta},D_x\psi(\overline t, \overline x),D_{xx}^2\psi(\overline t, \overline x),0).
\end{multline*}
\end{claim}

By definition of viscosity subsolution of \eqref{HJBSP} and using the regularity of $w$ (proved in 
Theorem \ref{regw}) , $\psi$ and $\chi$
:
\begin{multline*}%\label{subpsi}
-\partial_t\psi(t_{\epsilon, \eta},x_{\epsilon, \eta})+H(x_{\epsilon, \eta},y_{\epsilon, \eta},D_x\psi(t_{\epsilon, \eta},x_{\epsilon, \eta}),D_{xx}^2\psi(t_{\epsilon, \eta},x_{\epsilon, \eta}),0)\\
 - {\cal L}(y_{\epsilon, \eta}, D_yw(y_{\epsilon, \eta}),D^2_{yy}w(y_{\epsilon, \eta}))-\eta {\cal L}(y_{\epsilon, \eta}, D_y\chi(y_{\epsilon, \eta}),D^2_{yy}\chi(y_{\epsilon, \eta}))\\
+a V^\epsilon(t_{\epsilon, \eta},x_{\epsilon\eta},y_{\epsilon, \eta})\leq 0 
\end{multline*}
Now we use \eqref{cell} for $y=y_{\epsilon, \eta}$
\begin{multline*}%\label{subpsi1}
-\partial_t\psi(t_{\epsilon, \eta},x_{\epsilon, \eta})+H(x_{\epsilon, \eta},y_{\epsilon, \eta},D_x\psi(t_{\epsilon, \eta},x_{\epsilon, \eta}),D_{xx}^2\psi(t_{\epsilon, \eta},x_{\epsilon, \eta}),0)\\
 + {\overline H}(\overline x,D_x\psi(\overline t, \overline x), D_{xx}^2\psi(\overline t, \overline x))-H(\overline x,y_{\epsilon, \eta},D_x\psi(\overline t, \overline x),D_{xx}^2\psi(\overline t, \overline x),0)\\
-\eta {\cal L}(y_{\epsilon, \eta}, D_y\chi(y_{\epsilon, \eta}),D^2_{yy}\chi(y_{\epsilon, \eta}))
+a V^\epsilon(t_{\epsilon, \eta},x_{\epsilon\eta},y_{\epsilon, \eta})\leq 0 
\end{multline*}
which is equivalent to \eqref{claim3}.

\begin{claim}
\label{claimybdd}
If $\ds t_{\epsilon, \eta}\in]\frac{\bar t}{2},\frac{\bar t+T}{2}[$ , then $\{y_{\epsilon, \eta}\}_{\epsilon}$ is uniformly bounded.
\end{claim}

By contradiction: let us assume that there exists a sequence $\{\epsilon_n\}_n$ with $\ds \epsilon_n\rightarrow 0$ such that
\[
 t_{\epsilon_n, \eta}\in]\frac{\bar t}{2},\frac{\bar t+T}{2}[, \quad |y_{\epsilon_n, \eta}|\rightarrow \infty \hbox{  for }n\rightarrow \infty
\]
Then, there holds, calling  $y_{\epsilon_n, \eta}=(y_{\epsilon_n, \eta, 1},y_{\epsilon_n, \eta, 2}, y_{\epsilon_n, \eta, 3})$:
\begin{multline}\label{claimcontr1}
-{\cal L}(y_{\epsilon_n, \eta}, D_y\chi(y_{\epsilon_n, \eta}),D_{yy}\chi(y_{\epsilon_n, \eta}))=\\ 2(-2+(k_1-4)y^2_{\epsilon_n, \eta,1}+(k_2-4)y^2_{\epsilon_n, \eta,2}+k_3y^2_{\epsilon_n, \eta,3})\rightarrow \infty.
\end{multline}
Moreover, by Claim~\ref{claimbddtx} and
~\ref{regH1}
we get
\[\vert {\cal{F}}_{\epsilon_n, \eta}\vert \leq K(\eta)=K.\]

Coming back to \eqref{claim3}, using again Claim~\ref{claimbddtx} and the uniform boundedness of $V^\epsilon$ we have a contradiction.

\begin{claim}\label{convmax}
There holds
\begin{equation*}%\label{eqconvmax}
(t_{\epsilon, \eta},x_{\epsilon, \eta})\rightarrow ( \bar t,\bar x) \quad \hbox{ for } \epsilon\rightarrow 0.
\end{equation*}
\end{claim}

There exists $\ds ( \tilde t,\tilde x)\in[\frac{\bar t}{2},\frac{\bar t+T}{2}]\times \re^n$ such that (possibly passing to a subsequence)
\[
(t_{\epsilon, \eta},x_{\epsilon, \eta})\rightarrow ( \tilde t,\tilde x) \qquad \textrm{as } \epsilon\to 0.
\]
By definition of $\ds(t_{\epsilon, \eta},x_{\epsilon, \eta},y_{\epsilon, \eta})$ we have: $\ds \forall (t', x', y')\in[\frac{\bar t}{2},\frac{\bar t+T}{2}]\times \re^n\times \re^3$
\begin{multline}\label{claimvepmax}
V^\epsilon(t_{\epsilon, \eta},x_{\epsilon, \eta},y_{\epsilon, \eta})-\psi(t_{\epsilon, \eta},x_{\epsilon, \eta})-\epsilon(w(y_{\epsilon, \eta})+\eta\chi(y_{\epsilon, \eta}))\geq   \\
 V^\epsilon(t', x', y')-\psi(t', x')-\epsilon(w(y')+\eta\chi(y'))
\end{multline}
Let $ \ds (t, x, y)\in[\frac{\bar t}{2},\frac{\bar t+T}{2}]\times \re^n\times \re^3$. Passing to $\ds\limsup_{(\epsilon, t', x', y')\rightarrow (0, t, x, y)}$  in the previous inequality \eqref{claimvepmax}:
\begin{multline}\label{claimvmax}
\limsup_{(\epsilon, t', x', y')\rightarrow (0, t, x, y)}\left[V^\epsilon(t_{\epsilon, \eta},x_{\epsilon, \eta},y_{\epsilon, \eta})-\psi(t_{\epsilon, \eta},x_{\epsilon, \eta})-\epsilon(w(y_{\epsilon, \eta})+\eta\chi(y_{\epsilon, \eta}))\right]\geq   \\
 \bar V(t, x)-\psi(t, x)
\end{multline}
Moreover there is a sequence $\epsilon_n\rightarrow 0$ such that 
\begin{multline}\label{claimsubseq}
\limsup_{(\epsilon, t', x', y')\rightarrow (0, t, x, y)}\left[V^\epsilon(t_{\epsilon, \eta},x_{\epsilon, \eta},y_{\epsilon, \eta})-\psi(t_{\epsilon, \eta},x_{\epsilon, \eta})-\epsilon(w(y_{\epsilon, \eta})+\eta\chi(y_{\epsilon, \eta}))\right]=\\
\lim_{n\rightarrow\infty}
\left[V^{\epsilon_n}(t_{\epsilon_n, \eta},x_{\epsilon_n, \eta},y_{\epsilon_n, \eta})-\psi(t_{\epsilon_n, \eta},x_{\epsilon, \eta})-\epsilon_n(w(y_{\epsilon, \eta})+\eta\chi(y_{\epsilon_n, \eta}))\right]
\end{multline}
We split the argument according to the case that $t_{\epsilon_n, \eta}$ belongs to the interior or to the boundary of $[\frac{\bar t}{2},\frac{\bar t+T}{2}]$.
\begin{itemize}
\item If $\ds t_{\epsilon_n, \eta}=\frac{\bar t}{2}$ or $\ds t_{\epsilon_n, \eta}=\frac{\bar t+T}{2}$ we apply claim \ref{claimtbord} 
\begin{multline*}%\label{claimsubseq1}
V^{\epsilon_n}(t_{\epsilon_n, \eta},x_{\epsilon_n, \eta},y_{\epsilon_n, \eta})-\psi(t_{\epsilon_n, \eta},x_{\epsilon, \eta})-\epsilon_n(w(y_{\epsilon, \eta})+\eta\chi(y_{\epsilon_n, \eta})=\\ \Psi(t_{\epsilon_n, \eta},x_{\epsilon_n, \eta},y_{\epsilon_n, \eta})\leq -1+\epsilon_n k_0(1+\log(\eta))
\end{multline*}
If $\ds t_{\epsilon_n, \eta}=\frac{\bar t}{2}$ or $\ds t_{\epsilon_n, \eta}=\frac{\bar t+T}{2}$ for an infinite sequence of indices $\epsilon_n$ we have:
\[
-1\geq \bar V(t, x)-\psi(t, x)
\]
and this is a contradiction since , for $(t,x)=(\bar t, \bar x)$, $\bar V(\bar t, \bar x)-\psi(\bar t, \bar x)=0$.
\item If $\ds t_{\epsilon, \eta}\in]\frac{\bar t}{2},\frac{\bar t+T}{2}[$, then 
by Claim \ref{claimybdd} $\{y_{\epsilon, \eta}\}_{\epsilon}$ is uniformly bounded and we can assume (possibly passing to a subsequence) that 
\begin{equation*}%\label{yeta}
y_{\epsilon_n,\eta}\rightarrow \tilde y_\eta
\end{equation*}
Using \eqref{claimvmax}, \eqref{claimsubseq} and the upper-semicontinuity of $\bar V$
\begin{eqnarray*}%\label{claim6}
 \bar V(\tilde t, \tilde x)-\psi(\tilde t, \tilde x)&\geq&
\limsup_{(\epsilon', t', x', y')\rightarrow (0, \tilde t, \tilde x, \tilde y_\eta)}\left[V^{\epsilon'}( t', x', y')\right.\\ &&\qquad-\left. \psi( t', x')-\epsilon'(w(y')+\eta\chi(y'))\right]\\ &\geq& \lim_{n\rightarrow\infty}\left[
V^{\epsilon_n}(t_{\epsilon_n, \eta},x_{\epsilon_n, \eta},y_{\epsilon_n, \eta})\right.
\\ &&\qquad \left.-\psi(t_{\epsilon_n, \eta},x_{\epsilon, \eta})-\epsilon_n(w(y_{\epsilon, \eta})+\eta\chi(y_{\epsilon_n, \eta}))\right]\\ &\geq & \bar V(t, x)-\psi(t, x).
\end{eqnarray*}
Using the strict maximum property of $(\bar t, \bar x)$, we get  $(\bar t, \bar x)=(\tilde t, \tilde x)$.
\end{itemize}
Let us remark that the previous inequalities imply also:
\begin{equation}
\label{VepVbar}
\lim_{n\rightarrow\infty}
V^{\epsilon_n}(t_{\epsilon_n, \eta},x_{\epsilon_n, \eta},y_{\epsilon_n, \eta})=\bar V(\bar t, \bar x)
\end{equation}

\begin{claim} There holds
\begin{multline}\label{claim7}
-\partial_t\psi(\overline t, \overline x)+
\overline H\left(\overline x, D_x\psi(\overline t, \overline x), D_{xx}^2\psi(\overline t, \overline x)\right)\\
-\eta {\cal {L}}(\tilde y_\eta,D_y\chi(\tilde y_\eta), D^2\chi(\tilde y_\eta))
+a\bar V(\bar t, \bar x) \leq 0.
\end{multline}
\end{claim}

Using Claim \ref{eqVep}, we get
\begin{multline*}%\label{claim3n}
-\partial_t\psi(t_{\epsilon_n, \eta},x_{\epsilon_n, \eta})+
\overline H\left(\overline x, D_x\psi(\overline t, \overline x), D_{xx}^2\psi(\overline t, \overline x)\right)\\
-\eta {\cal {L}}(\tilde y_\eta,D_y\chi(y_{\epsilon_n, \eta}), D^2_{yy}\chi(y_{\epsilon_n, \eta}))+{\cal{F}}_{\epsilon_n, \eta}
+a V^{\epsilon_n}(t_{\epsilon_n, \eta},x_{\epsilon_n,\eta},y_{\epsilon_n, \eta})\leq 0.
\end{multline*}
Thanks to the regularity properties of $H$ (see \eqref{regH1}) and $\psi$, it is easy to get
\[
\ds {\cal{F}}_{\epsilon_n, \eta}\rightarrow 0
\]
From \eqref{VepVbar}, the statement follows easily.

\begin{claim} There holds
\label{claimend}
\begin{equation*}%\label{claimend1}
-\partial_t\psi(\overline t, \overline x)+
\overline H\left(\overline x, D_x\psi(\overline t, \overline x), D_{xx}^2\psi(\overline t, \overline x)\right)
+a\bar V(\bar t, \bar x) \leq 0.
\end{equation*}
\end{claim}

We split the argument according to the fact that $\{\tilde y_\eta\}_\eta$ is uniformly bounded or not.
If $\{\tilde y_\eta\}_\eta$ is uniformly bounded, then passing to the limit as $\eta\to 0$  in inequality~\eqref{claim7} we get the statement.
If $\{\tilde y_\eta\}_\eta$ is unbounded, eventually passing to a subsequence, we can assume that $|\tilde y_\eta|\to +\infty$ as $\eta\to 0$. Arguing as in~\eqref{claimcontr1}, we get
\[
-{\cal L}(\tilde y_\eta, D_y \chi(\tilde y_\eta), D^2_{yy}\chi(\tilde y_\eta))\to +\infty \qquad \textrm{as }\eta \to0.
\]
In particular, for $\eta$ sufficiently small, there holds
\[
-{\cal L}(\tilde y_\eta, D_y \chi(\tilde y_\eta), D^2_{yy}\chi(\tilde y_\eta))\geq0.
\]
Replacing this inequality in relation~\eqref{claim7}, we get the statement.

\noindent{\bf Conclusion of the proof.} By the arbitrariness of the test function~$\psi$ and of the point $(\overline t, \overline x)$, we obtain that $\overline V$ is a (viscosity) subsolution of the parabolic equation in~\eqref{EFFSP}.
\end{proof}

\subsection{Step 3}
\begin{proposition}
There holds
$$\overline V(T,x)\leq \overline g(x)\leq \underline V(T,x) \quad\forall x\in \re^n.$$
\end{proposition}
\begin{proof}
We shall prove only the former inequality since the latter is analogous. We first recall from Lemma~\ref{Vindy} that, for $t>0$, the function~$\overline V$ is independent of~$y$; hence, also $\overline V(T,x,y)$ is independent of~$y$.
Fix $\bar x\in \re^n$; for every $r>0$ sufficiently small, we define
\begin{equation*}
g^r(y):=\sup_{|x-\bar x|\leq r}g(x,y)
\end{equation*}
and we observe that assumption (A4) ensures that $g^r$ is a continuous bounded function with
\begin{equation*}%\label{stab3}
|g^r(y)-g(x,y)|\leq \omega(r)\qquad \forall x\in B_r(\bar x).
\end{equation*}
This implies that 
\begin{equation}\label{stab31}
-\omega(r)+g^r(y)\leq g(\bar x,y)\leq \omega(r)+g^r(y) \qquad \forall y\in\re^3.
\end{equation}
We also introduce the parabolic Cauchy problem
\begin{equation}\label{stab1}
\left\{\begin{array}{ll}
\partial_t w^r- tr(\sigma\sigma^T D_{yy}^2 w^r)-b\cdot D_y w^r=0&\qquad\textrm{in }(0,\infty)\times \re^3\\
w^r(0,y)=g^r(y)&\qquad\textrm{on } \re^3;
\end{array}\right.
\end{equation}
by standard arguments, it admits exactly one bounded solution.
On the other hand, let us also consider the problem
\begin{equation}\label{stab2}
\left\{\begin{array}{ll}
\partial_t w'- tr(\sigma\sigma^T D_{yy}^2 w')-b\cdot D_y w'=0&\qquad\textrm{in }(0,\infty)\times \re^3\\
w'(0,y)=g(\bar x,y)&\qquad\textrm{on } \re^3;
\end{array}\right.
\end{equation}
we recall from \cite[Theorem4.2]{MMT} that there holds $\ds\lim_{t\to\infty}w'(t,y)=\overline g(\bar x)$ locally uniformly in~$y$; in particular, for every $\eta>0$ and $R>1$, there exists $\tau>0$ such that
\begin{equation}\label{stab4}
|w'(t,y)-\overline g(\bar x)|\leq \eta \qquad \forall (t,y)\in (\tau,\infty)\times B_R(0).
\end{equation}
Moreover, by relation~\eqref{stab31}, one can easily show that $w^r(t,y) \pm \omega (r)$ are respectively a supersolution and a subsolution to problem~\eqref{stab2}; hence, the comparison principle yields
\begin{equation*}
|w'(t,y)-w^r(t,y)|\leq \omega (r)\qquad \forall (t,y)\in (0,\infty)\times \re^3.
\end{equation*}
By the last inequality and~\eqref{stab4} we deduce that: for every $\eta>0$ and $R>1$, there exists $\tau>0$ such that
\begin{equation}\label{stab5}
|w^r(t,y)-\overline g(\bar x)|\leq \eta+\omega (r) \qquad \forall (t,y)\in (\tau,\infty)\times B_R(0).
\end{equation}

For later use, we introduce some notations; we set $Q^r:=(T-r,T)\times B_r(\bar x)$ and let $M\in \re$ be sufficiently large that  for every $(t,x,y)\in(0,T)\times B_r(\bar x)\times \re^3$: 
\begin{eqnarray}
\label{stabM}
\|w'\|_\infty\leq M\\
 \|w^r\|_\infty\leq M\\
|V^\epsilon(t,x,y)|\leq M.
\end{eqnarray}
Consider also a smooth function $\psi_0=\psi_0(x)$ (namely, it is independent of $t$ and of $y$) such that
\begin{equation}\label{stab7}
\left\{\begin{array}{ll}\psi_0(\bar x)=\ds 0,\\
\psi_0(x) \geq 0\quad \forall x\in  B_r(\bar x),\\
\ds\psi_0(x)\geq M-\inf_{(z,y)\in B_r(\bar x)\times \re^3}g(z,y) \quad \forall x\in \partial B_r(\bar x).
\end{array}\right.
\end{equation}
Let $C>0$ be a constant such that
\begin{equation}
\label{stabC}
|H(x,y,D_x\psi_0(x), D_{xx}^2\psi_0(x),0)|\leq C\qquad \forall (x,y)\in B_r(\bar x)\times \re^3.
\end{equation}

For $(t,x,y)\in Q^r\times \re^3$, we define
\begin{equation*}
\psi^\epsilon (t,x,y):=w^r\left(\frac{T-t}{\epsilon},y\right)+\psi_0(x) +C_1(T-t)
\end{equation*}
with $C_1:=C+aM$.
We claim that the function~$\psi^\epsilon$ is a supersolution to the following initial-boundary value problem
\begin{equation}\label{stab6}
\left\{\begin{array}{ll}
(i)\quad-\partial_t \Psi +H(x,y,D_x\Psi, D_{xx}^2\Psi, D_{xy}^2\Psi/\sqrt{\epsilon})-\frac{1}{\epsilon} {\cal L}(y,D_y \Psi, D_{yy}^2\Psi)&\\
\qquad\qquad\qquad +a\Psi=0\qquad \textrm{in }Q^r\times \re^3&\\
(ii)\quad\Psi (T,x,y)=g(x,y) \qquad\textrm{on } B_r(\bar x)\times \re^3&\\
(iii)\quad \Psi(t,x,y)=M \qquad\textrm{on }(T-r,T)\times\partial B_r(\bar x)\times \re^3.&
\end{array}\right.
\end{equation}
Assume for the moment that this claim is true. On the other hand, the function~$V^\epsilon$ is a subsolution to problem~\eqref{stab6}; therefore, by comparison principle (see, for instance, \cite[Proposition 1 (proof)]{AB3}), we get
\begin{equation*}
V^\epsilon(t,x,y)\leq \psi^\epsilon(t,x,y)\qquad \forall (t,x,y)\in (T-r,T)\times B_r(\bar x)\times \re^3.
\end{equation*}
For $y\in B_{R/2}(0)$  ($R$ is defined in \eqref{stab5}) and $t\in(T-r,T)$, we get
\begin{eqnarray*}
\overline V(t,x,y)&=&\limsup_{\epsilon\to 0^+, t'\to t, x'\to x, y'\to y}V^{\epsilon}(t',x',y')\\
&\leq &\limsup_{\epsilon\to 0^+, t'\to t, x'\to x, y'\to y}\psi^{\epsilon}(t',x',y')\\
&\leq &\limsup_{\epsilon\to 0^+, t'\to t, y'\to y}w^r\left(\frac{T-t'}{\epsilon},y'\right) +
\limsup_{t'\to t, x'\to x}[\psi_0(x')+C_1(T-t')]\\
&\leq &\bar g(\bar x)+\eta+\omega(r) + \psi_0(x)+C_1(T-t)
\end{eqnarray*}
where the last inequality is due to relation~\eqref{stab5} (observe that definitely $y'\in B_R(0)$ and $\frac{T-t'}{\epsilon}>\tau)$ and to the continuity of~$\psi_0$.
Since $\overline V$ is independent of~$y$, we deduce
$$
\overline V(t,x)\leq \overline g(\bar x)+\eta+\omega(r) + \psi_0(x)+C_1(T-t).
$$
Passing to the $\ds\limsup_{t'\to T^-, x'\to \bar x, y'\to y}$ (recall $\psi_0(\bar x)=0$), we infer
$$
\overline V(T,\bar x)\leq \overline g(\bar x)+\eta+\omega(r).
$$
By the arbitrariness of $\eta$ and of $r$, we get
$$
\overline V(T,\bar x)\leq \overline g(\bar x)
$$
which is equivalent to our statement.

Let us now pass to prove the claim: $\psi^\epsilon$ is a supersolution to problem~\eqref{stab6}.
First we check the initial-boundary conditions $(ii)$ and $(iii)$.
In order to prove $(ii)$, it suffices to note that the definition of~$g^r$ and the second property in~\eqref{stab7} entail
$$
\psi^\epsilon(T,x,y)=w^r(0,y)+\psi_0(x)\geq w^r(0,y)\geq g(x,y) \qquad \forall (x,y)\in B_r(\bar x)\times \re^3.
$$
In order to prove $(iii)$, we observe that $\ds g^r(y)\geq \inf_{(x,z)\in B_r(\bar x)\times \re^3}g(x,z)$; hence, the comparison principle for \eqref{stab1} yields
$$
w^r(t,y)\geq \inf_{(x,y)\in B_r(\bar x)\times \re^3}g(x,y) \qquad \forall (t,y)\in(0,\infty)\times \re^3.
$$
Taking also into account the third property in~\eqref{stab7}, we conclude
$$
\psi^\epsilon(t,x,y)=w^r(\frac{T-t}{\epsilon},y)+\psi_0(x)+C_1(T-t)\geq M
$$
for every $(t,x,y)\in (T-r,T)\times\partial B_r(\bar x)\times \re^3$.\\
Now we prove (i). Let us assume that $w^r$ is a classical solution to~\eqref{stab1}. 
In this case, in $(T-r,T)\times B_r(\bar x)\times \re^3$ there holds
\begin{eqnarray*}
&&-\partial_t \psi^\epsilon +H(x,y,D_x\psi^\epsilon, D_{xx}^2\psi^\epsilon, D_{xy}^2\psi^\epsilon/\sqrt{\epsilon})-\frac{1}{\epsilon} {\cal L}(y,D_y \psi^\epsilon, D_{yy}^2\psi^\epsilon) +a\psi^\epsilon\\
&&=\frac{1}{\epsilon}\left[\partial_t w^r-{\cal L}(y,D_y w^r, D_{yy}^2w^r)\right]+C_1+H(x,y,D_x\psi_0, D_{xx}^2\psi_0, 0) +a\psi^\epsilon\\
&&=C_1+H(x,y,D_x\psi_0, D_{xx}^2\psi_0, 0) +a w^r(\frac{T-t}{\epsilon},y)+a\psi_0(x)+aC_1(T-t)\\
&&=C_1+a w^r(\frac{T-t}{\epsilon},y) +H(x,y,D_x\psi_0, D_{xx}^2\psi_0, 0) \\
&&\geq C_1-C-aM\\
&&\geq 0.
\end{eqnarray*}
where we used the definition of $C_1:=C+aM$, the definition of $C$ in \eqref{stabC} and the definition of $M$ \eqref{stabM}.\\
In the case when $w^r$ is only a viscosity solution to~\eqref{stab1}, we can accomplish the proof following the same arguments of~\cite[Theorem 3]{AB3}.
\end{proof}

\subsection{Step 4}
We proved that $\overline V$ and $\underline V$ are respectively a subsolution and a supersolution of \eqref{EFFSP}.
WE can apply the comparison principle to equation \eqref{EFFSP} that holds since $\overline H$ defines a degenerate elliptic equation and $H$ is Lipschitz continuous on $x$ (see \eqref{regH1}).
Then $\overline V(t,x)\leq \underline V(t,x)$; by definition the reverse inequality is obvious.
Then $V=\overline V(t,x)=\underline V(t,x)$ is the unique continuous solution of the parabolic equation \eqref{EFFSP} and the local uniform convergence follows from standards arguments.

\vskip1cm

\noindent{\sc Acknowledgments } \\
We thank Martino Bardi, Marco Bramanti and Olivier Ley for helpful discussions and suggestions.\\
The first and the second authors  are members of the INDAM-Gnampa and are partially supported by the research project of the University of Padova "Mean-Field Games and Nonlinear PDEs" and by the Fondazione CaRiPaRo Project "Nonlinear Partial Differential Equations: Asymptotic Problems and Mean-Field Games". 
The third author has been partially funded by the ANR project ANR-16-CE40-0015-01.

%\nocite{Fri:64}
%\bibliography{fp}
%\bibliographystyle{plain}
\vskip 0.3truecm

\noindent {\bf Address of the authors}\\
Paola Mannucci,\\
Dipartimento di Matematica "Tullio Levi-Civita",
Universit\`a degli Studi di Padova,
Via Trieste 63, 35131, Padova, Italy,\\
Claudio Marchi,\\
Dipartimento di Ingegneria dell'Informazione,
Universit\`a degli Studi di Padova,
Via Gradenigo 6/b, 35131, Padova, Italy,\\
Nicoletta Tchou, IRMAR, \\
Universit\'e de Rennes 1,
Campus de Beaulieu, 35042 Rennes Cedex, France \\
mannucci@math.unipd.it,\\ 
claudio.marchi@unipd.it,\\
nicoletta.tchou@univ-rennes1.fr

\end{document}